\documentclass[a4paper,reqno,12pt]{amsart}
\usepackage{a4wide}
\usepackage[dvips]{graphicx}
\usepackage{color}
\usepackage{amssymb}

\usepackage{amsmath}
\usepackage{amsfonts}
\usepackage{mathrsfs}
\usepackage{latexsym}
\usepackage{amscd}
\usepackage{graphicx}
\usepackage[utf8]{inputenc}
\usepackage[english]{babel}
\usepackage{enumerate}
\usepackage{afterpage}
\def\a{\alpha}
\def\b{\beta}
\def\d{\delta}

\def\g{\gamma}

\renewcommand{\to}{\longrightarrow}
\newcommand{\ov}{\overrightarrow}
\def\co{\colon\thinspace}
\newcommand{\Um}{\boldsymbol{\mu}}
\newcommand{\Ut}{\boldsymbol{\tau}}

\newcommand{\C}{\mathbb{C}}
\newcommand{\Z}{\mathbb{Z}}
\newcommand{\Q}{\mathbb{Q}}
\newcommand{\R}{\mathbb{R}}

\newcommand{\PGL}{\mathrm{PGL}}

\newcommand{\SLtwoC}{\mathrm{SL}(2,\C)}

\newcommand{\HH}{{\mathbb H}^2}

\newcommand{\X}{\mathfrak{X}}
\newcommand{\scc}{\mathscr{S}}

\DeclareMathOperator{\tr}{\mathrm{tr}}

\newtheorem{Theorem}{Theorem}[section]
\newtheorem{Lemma}[Theorem]{Lemma}
\newtheorem{Proposition}[Theorem]{Proposition}

\newtheorem{introthm}{Theorem}

\newtheorem{Definition}[Theorem]{Definition}
\newtheorem{Remark}[Theorem]{Remark}
\newtheorem{Example}[Theorem]{Example}

\title{Trace Systoles and Sink Constants}

\author{Fr\'{e}d\'{e}ric Palesi}
\address{Aix Marseille Universit\'{e}, CNRS, I2M, Marseille, France}
\email{frederic.palesi@univ-amu.fr}
\urladdr{www.i2m.univ-amu.fr/perso/frederic.palesi/}

\begin{document}

\begin{abstract}

Let $\Sigma$ be a surface with $\chi (\Sigma) < 0$, and a representation $\rho $ from the fundamental group $\pi_1 (\Sigma)$ into $ \rm{SL} (2 , \C)$. We define the \emph{trace systole} of $\rho$, denoted $\mathrm{tys} (\rho)$ as folows :
$$\mathrm{tys} (\rho) = \inf \left\{ | \tr (\rho (\gamma)) | \ , \ \gamma \in \pi_1 (S)  \mbox{ essential simple closed curve} \right\}$$
% the infimum of $|\tr (\rho (\gamma)) |$ for $\gamma$ a simple closed curve. %This provides a more algebraic generalisation of the usual systole of an hyperbolic surface. 
When $\Sigma$ is endowed with an hyperbolic structure, the trace systole of the holonomy representation is naturally related to the usual systolic length of the hyperbolic surface, which is one of the motivation for this study.
The function $\mathrm{tys}$ is bounded on relative character varieties of $\Sigma$, and in this article we compute explicitly the optimal bounds for the one-holed torus, the four-holed sphere and the non-orientable surface of genus $3$. The proofs rely on the correspondance between representations of these surface groups and so-called Markoff maps which were introduced by Bowditch. From this, we infer various consequences on the optimal systolic inequalities of certain hyperbolic manifolds and also on non-Fuchsian representations for these surfaces.

\end{abstract}

\maketitle

\section{Introduction}

Let $\Sigma$ be a surface of finite type with $\chi (\Sigma) < 0$, and denote by $\widehat{\Omega}$ the set of homotopy classes of unoriented essential simple closed curves on $\Sigma$. This set can be seen as a subset of $\pi_1 (\Sigma) / (x \sim x^{-1}) $. Motivated by the notion of systole of an hyperbolic surface, we define the \emph{trace systole} $\rm{tys} (\rho)$ of a representation $\rho \in  \mathrm{Hom} (\pi_1 (S) , \SLtwoC ) $ by :
$$ \rm{tys} (\rho) = \inf \left\{ | \tr (\rho (\gamma)) | \ , \ \gamma \in \widehat{\Omega}  \right\}.$$
As the trace of an element is invariant by conjugation, this map descends to a map, still denoted $\mathrm{tys}$, on the character variety $\X (\Sigma) = \mathrm{Hom} (\pi_1 (\Sigma) , \SLtwoC ) / \SLtwoC$, which is the quotient of the space of representations by the conjugation action of $G$.

When $\Sigma$ is a closed surface, classical results on systoles show that the function $\rm{tys}$ is bounded on $\X (\Sigma)$ and attains its maximum. So we can define :
$$\rm{Tys} (\Sigma) = \max \big\{ \mathrm{tys} (\rho) , [\rho] \in \X (\Sigma) \big\}$$

When $\Sigma$ has boundary components, the fundamental group $\pi_1 (\Sigma)$ is a free group, and $\rm{tys}$ is usually an unbounded function on the whole character variety. In that case, one has to consider relative character varieties, with prescribed traces on each boundary. More precisely, if $\Sigma$ has $p$ boundary components represented by $c_1 , \dots, c_p \in \pi_1 (\Sigma)$, and $\mathfrak{B} = (b_1,\dots, b_p) \in \C^p$, we can consider the $\mathfrak{B}$-relative character variety as
$$\X_\mathfrak{B} (\Sigma) = \big\{ [\rho] \in \X (\Sigma) \ \big| \ \forall i \in \{ 1 , \dots , p\} , \tr (\rho (c_i)) = b_i \big\}$$
Therefore, we define similarly $\rm{Tys}_\mathfrak{B} (\Sigma)$ as the maximum of the restriction of the map $\rm{tys}$ on $\X_\mathfrak{B} (\Sigma)$. 

The main purpose of this article is to find the explicit values of $\rm{Tys} (\Sigma)$ and $\rm{Tys}_{\mathfrak{B}} (\Sigma)$, when the surface $\Sigma$ is either the one-holed torus, the four-holed sphere or the non-orientable surface of genus $3$. To study the trace systole of a representation $\rho$ in these specific cases, we will simply study the map $\phi_\rho : \widehat \Omega \rightarrow \C$, defined by $\phi_\rho (X) = \tr (\rho (X))$, and we will use a particular combinatorial description of $\widehat \Omega$. More precisely, for each surface studied, the curve complex will correspond to the Farey graph on $\Q \cup \{ \infty \}$, or will be constructed naturally from it. One can then consider the dual graph $\mathcal{T}$ to this complex which is an infinite trivalent tree embedded in the disk, and the set $\Omega$ of complimentary regions is in direct correspondance with $\widehat{\Omega}$. In that setting, the vertices of $\mathcal{T}$ correspond to triples of curves with minimal intersection. 

So one can consider the map $\phi : \Omega \rightarrow \C$ independently of an underlying representation. Indeed this map will satisfy natural conditions around vertices and edges of $\mathcal{T}$ coming from the trace identities in $\mathrm{SL} (2 , \C)$, and depending on a parameter $\Um = (\lambda_1 , \lambda_2 , \lambda_3 , s) \in \C^4$. These conditions are as follows :
\begin{enumerate}
	\item If $X_1, X_2, X_3 \in \Omega$ are three regions meeting at a vertex $v \in V(\mathcal{T})$, then
$$	\phi(X_1)^2 + \phi(X_2)^2 + \phi(X_3)^2 - \phi(X_1)\phi(X_2) \phi (X_3) + \lambda_1 \phi (X_1) + \lambda_2 \phi(X_2) + \lambda_3 \phi (X_3) = s$$
	\item If $X_j, X_k \in \Omega$ intersect along edge $e_i \in E (\mathcal{T})$ and $X_i, X_i' \in \Omega$ are the two regions at the ends of $e_i$,  then
	$$\phi(X_i) + \phi(X_i') = \phi(X_j) \phi(X_k) - \lambda_i $$
\end{enumerate}
and are called respectively the vertex equation and the edge equation. A map satisfying these compatibility conditions at all vertices and edges will be called a \emph{(generalized) Markoff map}. 

The study of Markoff maps started with Bowditch \cite{bow_mar} who introduced them to give an alternative proof of McShane's identity and study quasifuchsian representations of the one-punctured torus. They were later studied by Tan-Wong-Zhang \cite{tan_gen} for any boundary conditions on the one-holed torus, and by Maloni-Palesi-Tan \cite{mal_ont} in the four holed sphere case. Similar objects have also been introduced  in the case of the three-holed projective plane by Huang-Norbury \cite{hua_sim} and Maloni-Palesi \cite{mal_ont2}. In all these situations, Markoff maps provide a method to construct open domain of discontinuity for  the mapping class group action and study the length spectrum of hyperbolic surfaces. 

%These depending on a parameter $\Um = (\lambda_1 , \lambda_2 , \lambda_3 , s) \in \C^4$. Such maps are called generalized Markoff mapsThese conditions come naturally when one interpret these maps as coming from representations $\rho \in \X_{\mathfrak{B}} (S)$ with $S$ a four-holed sphere. They were introduced by Bowditch \cite{bow_mar} and then later studied by Tan-Wong-Zhang \cite{tan_gen} and Maloni-Palesi-Tan \cite{mal_ont}. 

In this article, one of the motivation is to use these maps to give optimal systolic inequalities as follows. For a Markoff map $\phi$, one can define a natural orientation of the edges of the tree $\mathcal{T}$ from the values of $\phi$. Namely, the arrow on the edge joining two complimentary regions of $X, X' \in \Omega$ is oriented from $X$ to $X'$ when $| \phi (X) | > |\phi (X')|$. This allows one to describe an elementary \emph{trace reduction algorithm}, starting from a given vertex of the tree $\mathcal{T}$ and following the oriented edges, as this will always reduce the values of $\phi$ in the three regions around the considered vertex. One can hope that reducing the trace will eventually end up at the minimum of the map $\phi$, which will be related to the trace systole of the underlying representation. 

In general, a path in the graph that follows oriented edges will end at a \emph{sink} of the Markoff map, which is a vertex of $\mathcal{T}$ such that all three incident edges point towards that vertex. Using an elementary study of the inequalities satisfied around a sink, one can show that the minimal value of the Markoff map on the regions adjacent to that vertex is bounded above by a constant depending only on $\Um$. However, the optimal value of this constant, that we will call the \emph{sink constant} denoted $M (\Um)$, was unknown except for the $\Um = (0,0,0,0)$ case that was determined by Bowditch. The question of the value in more general cases was first asked by Tan-Wong-Zhang (\cite{tan_gen}) but remained entirely open since then. The first main result of this article is the exact value of $M (\Um)$ in the case where $\lambda_1 = \lambda_2 = \lambda_3  = 0$.

\begin{introthm}\label{introthm:sink1}
	Let $\mu \in \C$ and $\Um = (0, 0, 0, \mu)$. Then $M(\Um) = | t_\mu |$ where $t_\mu$ is a dominant root of the polynomial equation $X^3 - 3X^2 + \mu = 0$.
\end{introthm}

As Markoff maps in this case are in correspondance with representations of the one-holed torus surface group, one can then directly rely the sink constant $M(\Um)$ with the trace systole $\rm{Tys}_{\mathfrak{B}} (\Sigma_{1,1})$ for the one holed torus with some boundary data $\mathfrak{B}$. With some small modifications, one can also use this result to determine the trace systole for the non-orientable surface of genus $3$. So using the computed value of the sink constant above, one can obtain :

\begin{introthm}\label{introthm:sys1}
	\begin{enumerate} 
		\item Let $\Sigma_{1,1}$ be a one-holed torus and $k \in \C \setminus \{ 2 \}$. We have 
$$\mathrm{Tys}_k ( \Sigma_{1,1}  ) = | t_{k+2} |$$
		\item Let $N_3$ be the closed non-orientable surface of genus $3$.
		$$\mathrm{Tys} (N_3) = \sqrt{3 + \sqrt{17}} $$
	\end{enumerate}
\end{introthm}

For general $\Um \in \C^4$, the constant $M (\Um)$ appears to be much more difficult to determine. Nevertheless we can still obtain similar results when $\Um \in U= [0,+\infty[^3 \times ]- \infty, 4] \subset \R_+^4 $ and one restricts to Markoff maps with real positive image. These correspond to representations of the four-holed sphere $\Sigma_{0,4}$ where all four boundary traces are real and positive. In that case, we obtain : %Hence we obtain the following result on the trace systole for $\Sigma_{0,4}$ :

\begin{introthm}\label{introthm:sink2}
	Let $\Um = (\lambda_1 , \lambda_2, \lambda_3 , \mu) \in U$.  We denote by $T_{\Um}$ the largest positive real root of $$X^3 - 3X^2 - (\lambda_1 + \lambda_2+\lambda_3)X + \mu=0$$
	When restricted to Markoff maps with real positive image, we have : 
	$$M( \Um )= T_{\Um} $$
\end{introthm}

%\begin{introthm}\label{introthm:sink2}
%	Let $\mathfrak{B} = (a,b,c,d) \in \R_+^4$.   We denote by $T_\mathfrak{B}$ the largest positive real root of $$X^3 - 3X^2 - (ab+bc+cd+ac+bd+ad)X + (4-a^2-b^2-c^2-d^2-abcd)=0$$
%	Then we have 
%	$$\rm{Tys}_\mathfrak{B} (\Sigma_{0,4} ) = T_\mathfrak{B}$$
%\end{introthm}

When $S$ is an hyperbolic surface which is homeomorphic to $\Sigma$ then it gives rise to an holonomy representation $\rho_S : \pi_1 (\Sigma) \rightarrow \PGL (2 , \R)$ which is well-defined up to conjugacy. The length $l_S (\gamma)$ of a closed geodesic $\gamma$ on the hyperbolic surface $S$ is related to the trace of its image by the holonomy representation using the formula
$$l_S (\gamma) = \left\{ \begin{array}{ll} 2 \cosh^{-1} \left( \frac{\tr (\rho_S (\gamma))}{2} \right) & \mbox{ if } \gamma \mbox{ is 2-sided} \\
				2 \sinh^{-1} \left( \frac{\tr (\rho_S (\gamma))}{2} \right) & \mbox{ if } \gamma \mbox{ is 1-sided}
				\end{array}\right. $$
				
This gives a natural relation between the value of the trace systole of a Fuchsian representation and the usual systole $\mathrm{sys}(S)$ of the corresponding hyperbolic surface. This relation can be generalised for hyperbolic $3$-manifolds coming from quasi-Fuchsian representations of these surfaces, and also for incomplete hyperbolic structures on singular surfaces with conical singularities. Using this correspondance, one can obtain several systolic inequalities that we can sum up as follows :

\begin{introthm}\label{introthm:sys2}\

	\begin{enumerate}
		\item If $S$ is an hyperbolic one-holed torus, with geodesic boundary of length $l$, then
			$$\cosh \left( \frac{ sys(S) }{2} \right) \leq \cosh \left( \frac{l}{6} \right) + \frac 12 $$
		\item If $S$ is a quasi-Fuchsian structure of a once-punctured torus (with a cusp), then 
			$$\cosh \left( \frac{ sys(S) }{2} \right) \leq \frac 32 $$
		\item If $S$ is a singular hyperbolic structure on a torus with a conical singularity of angle $\theta$, then 
			$$\cosh \left( \frac{ sys(S) }{2} \right) \leq \cos \left( \frac{\theta}{6} \right) + \frac 12$$
		\item If $S$ is an hyperbolic four-holed sphere, with geodesic boundaries of length $l_i = 2 \cosh^{-1} \left( \frac{a_i}{2} \right)$, then 
			$$ \cosh \left( \frac{sys(S)}{2} \right) \leq \frac{T_\mu}{2}$$
			with $\Um = (a_1a_2 +a_3a_4 , a_1a_4 + a_2 a_3 , a_1 a_3 + a_2 a_4 , 4-a_1^2-a_2^2-a_3^2-a_4^2 - a_1 a_2 a_3 a_4)$.
		\item If $S$ is a quasi-Fuchsian hyperbolic structure of a four-punctured sphere, then :
			$$\cosh \left( \frac{sys(S)}{2} \right) \leq  \frac{7}{2} $$
		\item If $S$ is a quasi-Fuchsian structure on $N_3$, then
			$$\cosh (sys(S) ) \leq \frac{5+\sqrt{17}}{2}$$
	\end{enumerate}
	Moreover, all the inequalities above are optimal.
\end{introthm}
				
	The first two results were already known from previous works of Schmutz Schaller \cite{schmutz-systole} and Gendulphe \cite{gen_pay}, but the others appear to be new, at least in this form. In any case, even for (1) and (2), the proofs that we give here use a completely independent approach that does not rely on hyperbolic geometry. One can hope to generalize these results to compute optimal systolic inequalities for other hyperbolic surfaces of small complexity, which is still an open problem in many cases.
	
	\medskip
	
	We can also study properties of non-Fuchsian representations of surface groups using the notion of trace systole. In particular, this is related to a question of Bowditch \cite{bow_mar} : \emph{for a given type-preserving representation of a surface that is not discrete, does there exists a simple closed curve such that $\rho (\gamma)$ is not an hyperbolic element ?} Note that this would imply that the trace systole of such a representation is less than $2$. We prove that the answer is positive for $\Sigma_2$ the surface of genus $2$ and representations with Euler class $\pm 1$.
	
	\begin{introthm}\label{introthm:bow}
		Let $\rho : \pi_1 (\Sigma_2) \rightarrow \mathrm{PSL} (2 , \R)$ be a representation with Euler class $\pm 1$. Then there exists a simple closed curve $\gamma \in \pi_1 (\Sigma_2)$ such that $| \tr (\rho (\gamma)) | \leq 2$. 
	\end{introthm}

	This result was already proven by Marche and Wolff \cite{mar-wol}, using results on domination of non-Fuchsian representations by Fuchsian ones, and the explicit value of the Bers constant. The proof that we give here is completely independant and self-contained using our results on trace systoles of representations and we hope that this approach could be used for more general surfaces, where the answer to Bowditch's question is still unknown.
	
	\medskip

\medskip

\textbf{Plan of the paper.} 

We first recall the necessary background on Markoff maps in Section \ref{s:mar} focusing only on the combinatorial setting. The Section \ref{s:sink} is the core of the paper and is devoted to the definition of the the sink constant of Markoff maps, and the proofs of Theorems \ref{introthm:sink1} and \ref{introthm:sink2} providing explicit values of these constants in various cases. The proofs are technical but elementary and rely on finding the minimum of a function related to the Markoff map, on an explicit domain that is given by the inequalities defining a sink. In Section \ref{s:char} we will give the precise relation between character varieties of surface groups and Markoff maps, in the cases that we are interested in, namely the one-holed torus, the four-holed sphere and the non-orientable surface of genus $3$. This will allow us to give in Section \ref{s:sys} the main results in terms of trace systoles and prove Theorems \ref{introthm:sys1} and also get a geometrical interpretation in terms of the usual systole as described in Theorem \ref{introthm:sys2}. Finally, in Section \ref{s:bow} we will use these results to give an alternative proof of  Theorem \ref{introthm:bow} corresponding to Bowditch's question for the surface of genus $2$.

\section{Generalized Markoff maps}\label{s:mar}

In this section we recall the main definitions and properties of generalized Markoff maps, and refer to previous works \cite{bow_mar,tan_gen,mal_ont}  for more details.

\subsection{Farey triangulation and binary tree.}	\

	Let $\mathcal{F}$ be the Farey triangulation of the hyperbolic plane $\HH$. Recall that the ideal  vertices of $\mathcal{F}$ correspond to $\Q \cup \{ \infty \} \subset \partial \HH$ and that two vertices $\frac pq$, $\frac rs$ of $\mathcal{F}$ are joined by an edge if $|pq - rs| = 1$ (where we assume that $p,q,r,s \in \Z$ and $p\wedge q = r \wedge s = 1$). Let $\mathcal{T}$ be the dual graph to $\mathcal{F}$, where vertices correspond to the triangles of $\mathcal{F}$ and edges come from adjacency of triangles; see Figure \ref{fig:Farey}. We know that $\mathcal{T}$ is a countably infinite simplicial tree properly embedded in the plane all of whose vertices have degree $3$. We note $V(\mathcal{T})$ and $E(\mathcal{T})$ the set of vertices and edges of $\mathcal{T}$ respectively.

A \emph{complementary region} of $\mathcal{T}$ is the closure of a connected component of the complement, and we denote by $\Omega = \Omega (\mathcal{T}) $ the set of complementary regions of $\mathcal{T}$. The regions are in correspondance with vertices of $\mathcal{F}$ and hence they are indexed by elements of $\Q \cup \{ \infty \}$.

\begin{figure}[hbt]
\begin{minipage}[c]{0.45\linewidth}
\centering
\includegraphics[height=7 cm]{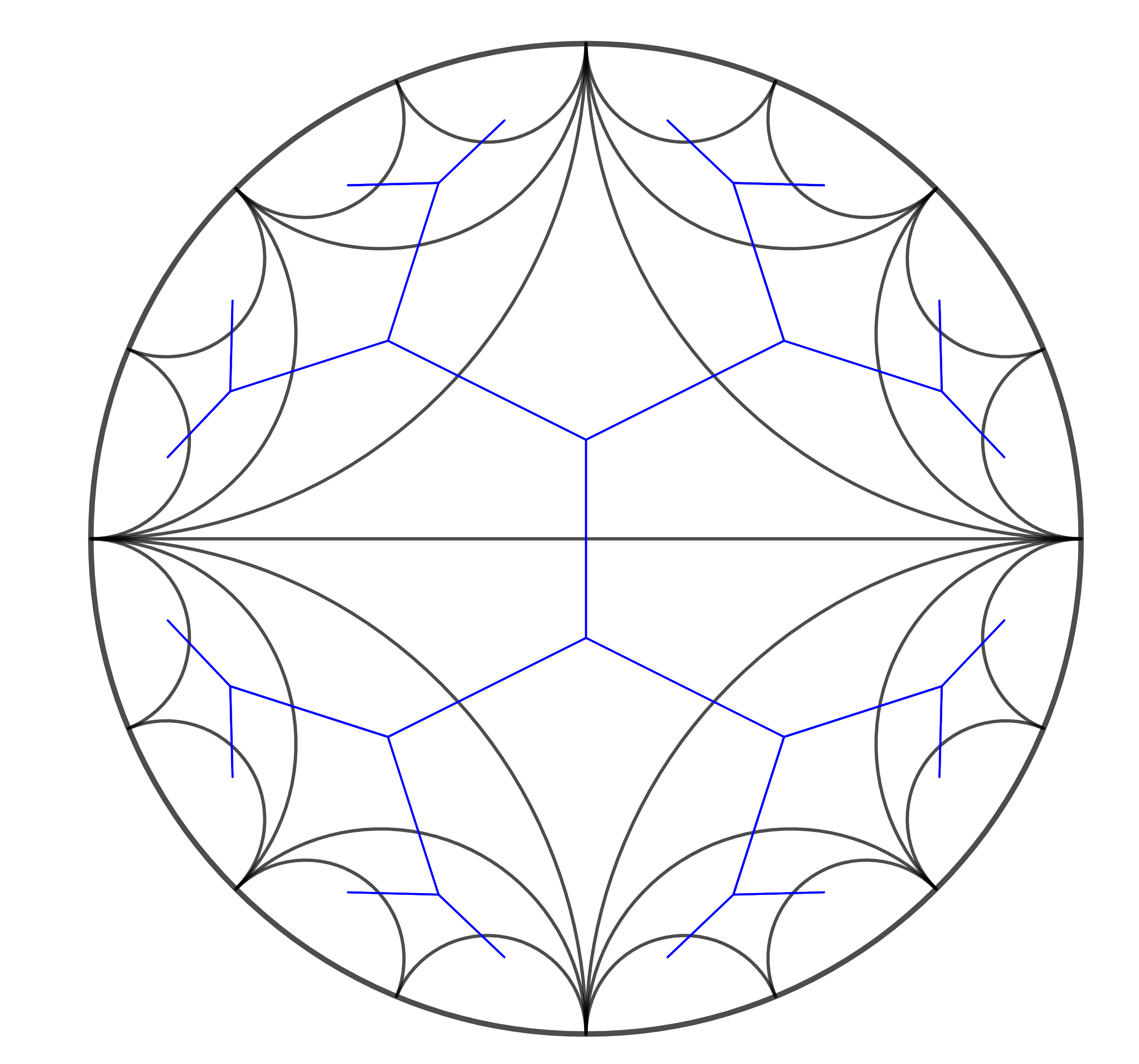}
\caption{The Farey triangulation (in black) and its dual graph $\mathcal{T}$ (in blue).}
\label{fig:Farey}
\end{minipage}
\begin{minipage}[c]{0.45\linewidth}
\centering
\includegraphics[height=7 cm]{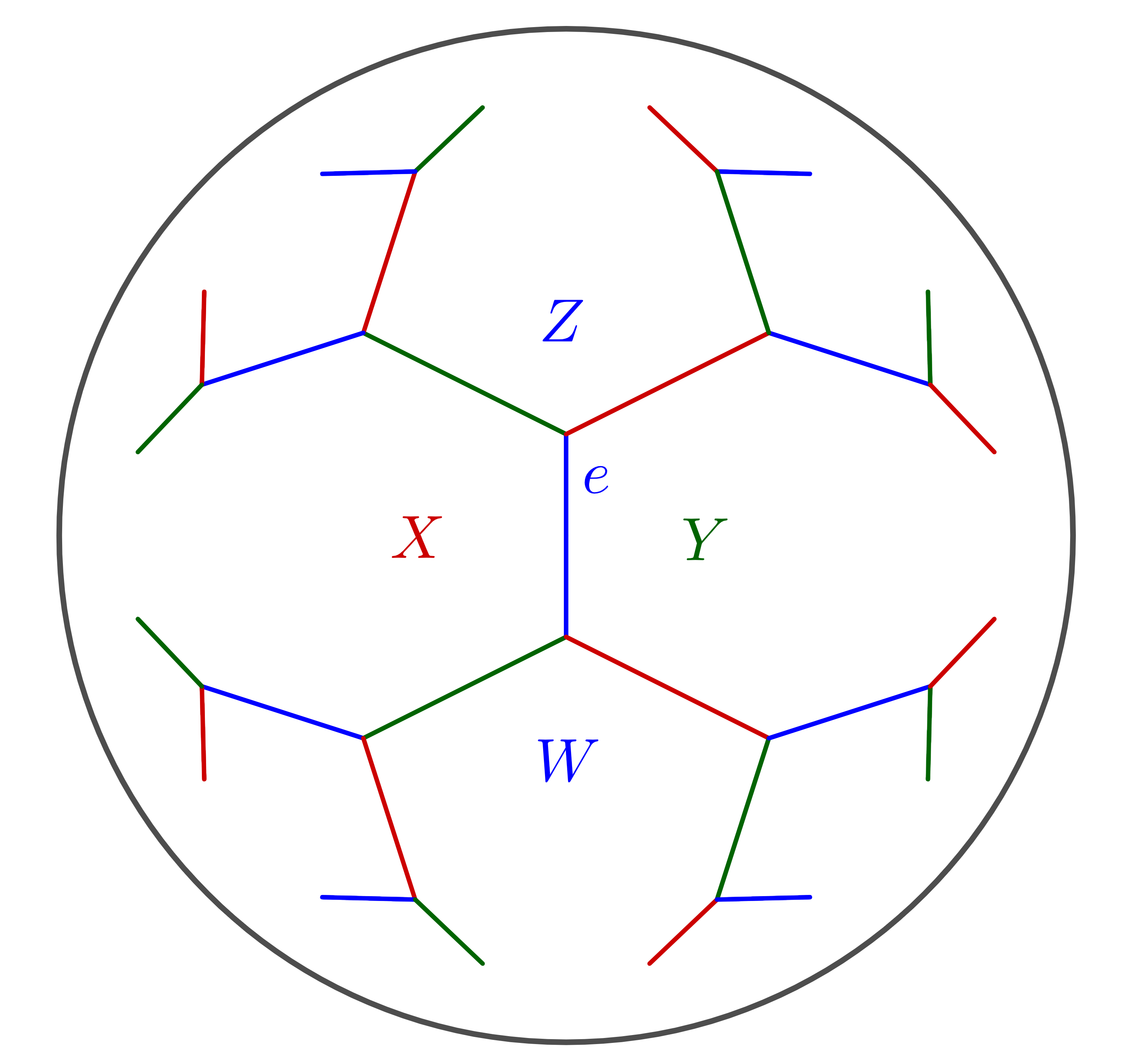}
\caption{A tricoloring of the regions and edges. Here $e\leftrightarrow (X,Y;Z,W)$.}
\label{fig:tricol}
\end{minipage}

\end{figure}

We will use capital letters $X,Y,Z,W, \ldots$ to denote elements of $\Omega$. For $e \in E(\mathcal{T})$, we will note $e\leftrightarrow (X,Y;Z,W)$ to indicate that $e=X \cap Y$ and $e\cap Z$ and $e \cap W$ are the endpoints of $e$; see Figure \ref{fig:tricol}.

We choose a tri-coloring of the regions and edges, namely a map $\mathcal{C} \co \Omega(\mathcal{T}) \cup E(\mathcal{T}) \to \{ 1, 2, 3 \}$ such that for any edge $e\leftrightarrow (X,Y;Z,W)$ we have $\mathcal{C}(e)=\mathcal{C}(Z) = \mathcal{C}(W)$ and such that $\mathcal{C} (e)$, $\mathcal{C}(X)$ , $\mathcal{C}(Y)$ are all different. This implies that the colors of three regions meeting at a vertex are all different, and the same also holds for the three edges meeting at a vertex. In fact, the coloring is completely determined by a coloring of the three regions (or three edges) around any specific vertex, and hence is unique up to a permutation of the set $\{1, 2, 3\}$. We denote by $\Omega_i = \Omega_i (\mathcal{T})$ the set of complementary regions with color $i$, and by $E_i = E_i (\mathcal{T})$ the set of edges with color $i$; see Figure \ref{fig:tricol}.

As a notational convention in the following, when $X,Y,Z$ are complementary regions around a vertex, we will consider that $X \in \Omega_1 $, $Y \in \Omega_2 $ and $Z \in \Omega_3 $, or in general that $X_i \in \Omega_i (\mathcal{T})$ .

\subsection{$\Um$--Markoff triples} \

For a complex quadruple $\Um = (\lambda_1, \lambda_2 , \lambda_3 , s) \in \C^4$, a $\Um$-{\it Markoff triple} is an ordered triple $(x_1,x_2,x_3) \in \C^3$ satisfying the $\Um$--Markoff equation, also called the vertex equation:
\begin{eqnarray}\label{eqn:vertex}
x_1^2 + x_2^2 + x_3^2 - x_1x_2x_3 + \lambda_1 x_1 + \lambda_2 x + \lambda_3 x_3 = s
\end{eqnarray}

\begin{Remark} Note that we slightly changed the convention used in our previous paper \cite{mal_ont} to ensure consistency with the notation of Tan-Wong-Zhang \cite{tan_gen}. To pass from one convention to the other, one simply has to replace $(x_1,x_2,x_3)$ by $(-x_1,-x_2,-x_3)$.

Note that, with this convention, if $(x_1,x_2,x_3)$ is a $\mu$--Markoff triple in the sense of Tan-Wong-Zhang, with $\mu \in \C$, then $(x_1,x_2,x_3)$ is a $\Um$--Markoff triple in our sense, with $\Um = (0,0,0,\mu)$.
\end{Remark}

It is easily verified that, if $(x_1,x_2,x_3)$ is a $\Um$--Markoff triple, so are the triples
\begin{equation}\label{eqn:elemoper}
(x_1,x_2,x_1x_2 - x_3 - \lambda_3), \hspace{0.3cm} (x_1,x_1x_3-x_2 - \lambda_2,x_3) \mbox{   and    } (x_2x_3 - x_1 - \lambda_1,x_2,x_3).
\end{equation}
It is important to note that in general, permutations triples are not $\Um$--Markoff triples, unlike the $\mu$--Markoff triples considered in Tan-Wong-Zhang \cite{tan_gen}. Namely, if $(x_1,x_2,x_3)$ is a $\Um$-Markoff triple, then the triple $(x_{\sigma(1)},x_{\sigma(2)},x_{\sigma(3)})$ with $\sigma \in \mathfrak{S}_3$ has no reason to be a $\Um$--Markoff triple.

\subsection{$\Um$--Markoff map}

\begin{Definition}
A $\Um$-{\it Markoff map} is a function $\phi \co \Omega \to \C$ such that
\begin{itemize}
\item[(i)] for every vertex $v \in V(\mathcal{T})$, the triple $(\phi(X_1), \phi(X_2), \phi(X_3))$ is a $\Um$--Markoff triple, where $X_1,X_2,X_3 \in \Omega$ are the three regions meeting $v$ such that $X_i \in \Omega_i$;
\item[(ii)] For any $i \in \{1,2,3\}$ and for every edge $e \in E_i(\mathcal{T})$ such that $e \leftrightarrow (X_j, X_k ; X_i, X_i')$ we have:
\begin{equation}\label{eqn:edge}
\phi(X_i)+\phi(X_i')=\phi(X_j) \phi(X_k)-\lambda_i.
\end{equation}
\end{itemize}
	We denote by ${\bf \Phi}_{\Um}$ the set of all $\Um$--Markoff maps.
\end{Definition}

One may establish a bijective correspondence between $\Um$--Markoff maps and $\Um$--Markoff triples, by fixing three regions $X_1, X_2, X_3$  which meet at some vertex $v_0$, and considering a map $\phi \mapsto (\phi(X_1), \phi (X_2 ), \phi(X_3))$. 

This process may be inverted by constructing a tree of $\Um$--Markoff triples as Bowditch did in \cite{bow_mar} for Markoff triples and as Tan, Wong and Zhang did in \cite{tan_gen} for the $\mu$--Markoff triples: given a $\Um$-Markoff triple $(x_1,x_2,x_3)$, set $\phi(X_i)=x_i$, and extend over $\Omega$ as dictated by the edge relations \ref{eqn:edge}. That's because if the edge relation \eqref{eqn:edge} is satisfied along all edges, then it suffices that the vertex relation \eqref{eqn:vertex} is satisfied at a single vertex to ensure that it is satisified at all vertices of $\mathcal{T}$.

This way, one obtains an identification of ${\bf \Phi}_{\Um}$ with the algebraic variety in $\C^3$ given by the $\Um$--Markoff equation. In particular, ${\bf \Phi}_{\Um}$ gets an induced topology as a subset of $\C^3$.

\subsection{Arrows assigned by a $\Um$--Markoff map} \

Let $\ov E(\mathcal{T})$ be the set of oriented edges.
Let $\phi \in {\bf \Phi}_{\Um}$. We can assign to each undirected edge, $e \in E (\mathcal{T})$, a particular directed edge, $\ov{e_\phi} \in \ov E (\mathcal{T})$, with underlying edge $e$, in the following way.

Suppose $e \leftrightarrow (X,Y;Z,W)$ and $\phi \in {\bf \Phi}_{\Um}$. If $|z|\geq|w|$, then we associate the element $\ov{e_\phi}$ in $\ov E(\mathcal{T})$, such that the arrow on $e$ points towards $W$; in other words, $\ov{e_\phi} = (X,Y;Z \rightarrow W)$. Reciprocally, if $|z|\leq |w|$, we put an arrow on $e$ pointing towards $Z$, that is, $\ov{e_\phi} = (X,Y;W \rightarrow Z)$. If it happens that $|z|=|w|$ then we consider that there is an arrow in both directions, as this will not affect the arguments in the latter part of this paper. 

For a given map $\phi$, a vertex with all three arrows pointing towards it is called a {\em sink}. Following previous works, we can also consider other types of vertex with respectively one, two or three arrows pointing away from it, and call them respectively a {\em merge}, a {\em fork} and  a \emph{source}, but we will only use the notion of sink in this article.

\section{Sink Constant}\label{s:sink}

This section will be devoted to the proofs of Theorems \ref{introthm:sink1} and \ref{introthm:sink2}. We start by recalling the following result from Maloni-Palesi-Tan (\cite{mal_ont}, Lemma 3.5) which will allow us to define what the sink constant is.

\begin{Lemma}\label{lem:sink}
  For all $\Um \in \C^4$, there exists a constant $m(\Um) \in \R_{>0}$ such that for all $\phi \in {\bf \Phi}_{\Um}$, if three regions $X_1, X_2, X_3$ meet at a sink, then
$$ \min \big\{ |\phi(X_1)| , |\phi (X_2)| , |\phi (X_3) | \big\} \leq m (\Um).$$
\end{Lemma}

An explicit value for $m (\Um)$ was given in \cite{mal_ont}, but this value was far from being optimal. So, a question that arises naturally is to understand the lowest possible value of this constant $m (\Um)$, which we call the \emph{sink constant} and denote $M (\Um)$. The question of the optimal value was first asked  in the $\Um = (0,0,0, \mu)$ case by Tan-Wong-Zhang in \cite{tan_gen}, who stated that it seemed difficult to determine that value. Previously, the only case known was the simplest one, studied by Bowditch who proved that  $M ((0,0,0,0)) = 3$ (see Lemma 3.2.(2) in \cite{bow_mar}).

%In this section, we fully answer the question of Tan-Wong-Zhang in the $(0,0,0,\mu)$ case, and also when $\Um$ belongs to an open subset of $\R^4$.

\subsection{Case $\lambda_1=\lambda_2=\lambda_3 = 0$}\

Here, we fully answer the question of Tan-Wong-Zhang in the $(0,0,0,\mu)$ case both for general complex valued Markoff maps, and also for real valued Markoff maps where a different bound appears.

	\subsubsection{Complex case}\
	
	The exact value of the sink constant will be directly related to the following implicit function:
	
	\begin{Definition}
		Let $a \in \C$. We denote by $t_a \in \C$ a dominant root of the polynomial equation $X^3 - 3X^2 + a = 0$.	
	\end{Definition}

	We first make a simple observation on the real part of $t_a$, denoted $\Re (t_a)$.

	\begin{Lemma}\label{lem:Retmu}
		For all $a \in \C$, we have $\Re (t_a) \geq 2$. Moreover, $\Re (t_a) = 2$ if and only if $a = 4$.
	\end{Lemma}
	\begin{proof}

		 Let $t_1, t_2, t_3$ be the three roots with multiplicity of the equation $X^3 - 3X^2 + a = 0$. From Vieta's formula, we have that  $t_1 + t_2 + t_3 = 3$ and $t_1^2 + t_2^2+t_3^2 = 9$. 
		 
		\underline{Claim} : If $x \leq y \leq z  \leq 2$ are such that $x + y + z = 3$, then $x^2+y^2+z^2 \leq 9$. Moreover, the last inequality is an equality if and only if $(x, y, z) = (-1 , 2 , 2)$.

	Indeed, as $x = 3 - y - z$ and $x$ is the smallest value, we easily get that $x \leq 1$. And as $y, z \leq 2$, we also have $x \geq -1$. Hence $x \in [-1, 1]$.
	
	 As $2z \geq y+z = 3-x$ we get $z \geq \frac{3-x}{2}$. If we let $x$ be fixed, and try to find the maximum of the function $f_x (z) =x^2 + (3-x-z)^2 + z^2$, on the interval $[\frac{3-x}{2} , 2 ]$ it is a straightforward computation to see that this maximum  is attained when $z=2$ and $\max f_x = 2x^2-2x+5$. The maximal value of the function $x \mapsto \max f_x$ on the interval $[-1, 1]$ is equal to $9$ and is attained only when $x= -1$. And if $x = - 1$ then $y=z = 2 $ which ends the proof of the claim.
	 
	 \medskip
	
		 Now, let $x_i = \Re (t_i)$, and assume by contradiction that $x_i < 2$ for all $i \in \{ 1, 2, 3 \}$. We can apply the previous result to the triple $(x_1, x_2, x_3)$ and we get $x_1^2 + x_2^2+x_3^2 <9$. But as $\Re (z^2) \leq (\Re (z)))^2$, we infer that $\Re (t_1^2 +t_2^2+t_3^2) \leq x_1^2+x_2^2+x_3^2 <9$, which is a contradiction. So at least one of the $t_j$ is such that $\Re (t_j) \geq 2$, which proves the first part of the claim.
		 
		 For the second part of the claim, if $a = 4$ then the solutions of the polynomial equation are $-1$ and $2$. And reciprocally, if $\max (\Re (t_j)) = 2$, then using the claim we know that $(\Re (t_1) , \Re (t_2), \Re (t_3)) = (-1, 2, 2)$ and hence $a = 3 t_1^2 - t_1^3 = 4$. 
		 
	\end{proof}

	We can now prove Theorem \ref{introthm:sink1} that we recall here.

	\begin{Theorem}\label{thm:sink1}
		Let $\mu \in \C$ and $\Um = (0,0,0,\mu)$, then 
		$$M (\Um) = |t_\mu|.$$
	\end{Theorem}

The theorem can be seen as a consequence of the following technical lemma :

	\begin{Lemma}\label{lem:sink1}
		The minimum of the function $f(p,q,r) = |pq|$ on the domain
		$$\mathcal{D}_\mu = \left\{ (p,q,r) \in \C^3 \ , \ |p|\geq |q| \geq |r|, \Re (p) , \Re (q), \Re (r) \leq \frac 12, p+q+r-1 = \mu pqr \right\} $$
		occurs for $p=q=r$. 
	\end{Lemma}

	\begin{proof}[Proof of Theorem \ref{thm:sink1} using Lemma \ref{lem:sink1}]
		
		Let $\phi \in {\bf \Phi}_{\Um}$, and assume that $X_1,X_2,X_3 \in \Omega$ are three regions meeting at a sink for $\phi$. We denote $x_i = \phi (X_i) $. Then, the directions of the three arrows incident to the sink give the following three inequalities : %, called the \emph{sink conditions} :
	\begin{equation}\label{eqn:sink}
		|x_1|\leq|x_2x_3-x_1|, \quad |x_2| \leq |x_1x_3 - x_2|, \quad |x_3| \leq |x_1x_2-x_3|
	\end{equation}
	Without loss of generality, we can assume that $|x_1| \geq |x_2| \geq |x_3|$. We now have to prove that $|x_3| \leq |t_\mu|$. We can assume that $x_3 \neq 0$, for otherwise $\min_i \{ |x_i| \} = 0$. 
	
	We make a change of variables to use the Lemma \ref{lem:sink1}. We set $z_1 = \frac{x_1}{x_2x_3}$, $z_2 = \frac{x_2}{x_3x_1}$ and $z_3 = \frac{x_3}{x_1x_2}$. The vertex equation (\ref{eqn:vertex}) becomes 
	\begin{equation}\label{eqn:vertexpqr} z_1+z_2+z_3-1 = \mu z_1z_2z_3 \end{equation}
	These new variables also satisfy $|z_1| \geq |z_2| \geq |z_3|$ and $|z_1z_2| = \frac{1}{|x_3|^2}$. And the inequalities  (\ref{eqn:sink}) now become $\Re (z_1) \leq \frac 12$, $\Re (z_2) \leq \frac 12$, and $\Re (z_3) \leq \frac 12$.
		
		So the triple $(z_1,z_2,z_3)$ is in $\mathcal{D}_\mu$, hence $|z_1z_2|$ is larger than the minimum of  $f(p,q,r)$ on the domain $\mathcal{D}_\mu$. Lemma \ref{lem:sink1} implies that this minimum is obtained when $p=q=r$. This means that this minimum is equal than the smallest root $\tau_\mu$ of the polynomial equation $3X - 1 - \mu X^3 = 0$. A simple change of variable in that equation shows that $\tau_\mu = \frac{1}{t_\mu}$.
	
	So we have $\frac{1}{|x_3|^2} = |z_1 z_2| \geq \frac{1}{|t_\mu|^2}$, and hence  $|x_3| \leq |t_\mu|$ as wanted, so we know that $M( \Um) \leq | t_\mu |$. 
	
	Moreover, the $\Um$-triple $(t_\mu, t_\mu , t_ \mu)$ defines a $\Um$-Markoff map where the initial vertex is a sink. Indeed, as $|t_\mu| \geq 2$ from Lemma \ref{lem:Retmu}, we get that $|t_\mu^2 - t_\mu| \geq ||t_\mu|^2 - |t_\mu|| \geq  |t_\mu|$. Hence $M(\Um) \geq |t_\mu |$, which ends the proof of the Theorem.
	\end{proof}
	
	Before getting to the proof of Lemma \ref{lem:sink1} which is rather technical, we start with the following remark. Let $\mu \neq 4$ and $\tau_\mu$ is the smallest root (in modulus) of the polynomial equation $3X - 1 - \mu X^3 = 0$. As $\tau_\mu = \frac{1}{t_\mu}$ and $\Re (t_\mu) > 2$, we have that $\tau_\mu$ is included in the disk of diameter $[0, \frac 12]$ and hence $\Re (\tau_\mu) \leq \frac 12$. This means that the triple $(\tau_\mu, \tau_\mu, \tau_\mu)$ is in $\mathcal{D}_\mu$, hence we know that the minimum of the function $f(p,q,r) = |pq|$ on the domain $\mathcal{D}_\mu$ is less or equal to $|\tau_\mu|^2$.

	\begin{proof}[Proof of Lemma \ref{lem:sink1}]\	
	Let $(p,q,r) \in \mathcal{D}_\mu$  realizing the minimum of the function $f$ on the domain $\mathcal{D}_\mu$. We will prove the equality $p=q=r = \tau_\mu$ with several intermediate steps.

		\begin{enumerate}
			\item Step 1 : We show that $|q| = |r| < \frac 12$.
			
			The previous remark shows that $|pq| \leq |\tau_\mu|^2$. As $|p| \geq |q|$, we have directly that $|q| \leq |\tau_\mu| < \frac 12$. %Note that for any $z \in \C$ such that $|z| < \frac 12$, we have $\Re (z) < \frac 12$

	Assume by contradiction that $|q| > |r|$. Then for $\epsilon >0$, we consider $p' = p'(\varepsilon) = (1-\varepsilon)p$ and $q' = q' (\varepsilon) = (1-\varepsilon)q$. We let $r' = r' (\varepsilon)$ be the unique complex number such that $p'+q'+r' - 1 = \mu p'q'r'$.  By continuity of $r'$ with respect to $\varepsilon$, we get that for $\varepsilon$ small enough, we have $(p',q',r') \in \mathcal{D}$ and $|p'q'|<|pq|$, which contradicts the minimality. Hence $|q| = |r|$ and this ends the proof of this first step.

	%So we have $|q| = |r|$. 

			\item Step 2 : We show that $q = r$.
			
			 Let $0<R < \frac 12 $ such that $|q| = |r| = R$. From the equation (\ref{eqn:vertexpqr}) we have that $p = \frac{1-q-r}{1- \mu qr}$.
			 
			 Hence we can study the function  $\zeta(x,y) = \dfrac{1-x-y}{1-\mu xy}$ on the set 
			 $$K = \left\{ (x,y) \in\C , |x| = |y| = R \right\} .$$ 
			 
			From step 1, it is clear that the minimum of $|\zeta|$ on the domain $K$ is equal to $|p|$, so we will show that this minimum is attained when $x = y$.
			
			 The image $\zeta(K)$ is a compact set in $\C$. As $|x|,|y| < \frac 12$ we see that this image avoids $0$, hence by the maximum principle, the minimum of $|\zeta|$ occurs on the boundary of $\zeta (K)$. The partial functions $\zeta(x, \cdot)$ and $\zeta(\cdot, y)$ are Möbius maps.  This means that for $x$ fixed, the image of $\zeta(x,\cdot)$ of the circle of radius $R$, is also a circle denoted $\mathcal{C}_x$. A direct computation shows that $\mathcal{C}_x$ is disjoint from the circle of radius $|\tau_\mu|$ centered at $0$, so we have that for all $(x,y) \in K$,  $|f(x,y)| \geq |\tau_\mu| \geq R $.

By symmetry of the two partial functions, the image of the other partial function $ \zeta (\cdot , x) $ is also the same circle $\mathcal{C}_x$. If a point $P$ is on the boundary of $\zeta (K)$, there exists $(x,y)$ in $K$ such that $P$ is in $\mathcal{C}_x$ and $\mathcal{C}_y$. These two circles cannot intersect transversely as otherwise the point $P$ would be in the interior of $\zeta (K)$. Hence the two circles are tangent, and this necessarily implies that the two circles are equal. This implies that the boundary of $\zeta (K)$ is exactly the set $\{ \zeta (x,x) , |x| = R \}$. As the minimum of $|\zeta (x,y)|$ on $K$ is attained on the boundary, it means that it is attained when $x=y$. This proves that $q = r$ and ends this second step.

			\item Step 3 : We show that $|q| = |\tau_\mu|$.
			
			From the previous step, we have that $q = r$ and hence  $pq = \frac{1-q-r}{1-\mu qr} q = \frac{1-2q}{1-\mu q^2}q$. 
			
			So we study the function $F(y) = y \dfrac{1-2y}{1-\mu y^2} = y \zeta (y,y) $ on the domain 
			$$K' = \left\{ y \in \C \ , \  |y| \leq |\tau_\mu| , \Re \left( \frac{1-2y}{1-\mu y^2} \right) \leq \frac 12 \right\}$$

			 From the previous discussion, the minimum of $|F|$ is precisely the minimum of $f$ on $\mathcal{D}$.			 Again, this function $F$ is bounded away from $0$ on $K'$, so by the maximum principle, the minimum of $|F|$ is attained on one of the boundary of the domain $K'$. So we study the minimum of $|F|$ for each of the two conditions defining the boundary.
			 
			 As $F(\tau_\mu) = \tau_\mu^2$ we can already see that the minimum of $|F(y)|$ on the boundary defined by $|y|=|\tau_\mu|$ is less than or equal to $|\tau_\mu|^2$.

			So, it is sufficient to prove the minimum of $|F|$ on the set  
			$$K'' = \left\{ y \in \C \ , \ \Re \left( \frac{1-2y}{1-\mu y^2} \right) = \frac 12 \right\}$$			
is greater than $|\tau_\mu|^2$.

Let $y \in K''$. There exists a real number $\beta$ such that $\frac{1-2y}{1-\mu y^2} = \frac 12 + i \beta$. So we can express $y$ in terms of $\beta$ as follows.
			$$y_\pm (\beta) = \dfrac{1 \pm \sqrt{1 - \mu (\frac 12 + i \beta)+ \mu (\frac 12 + i \beta)^2}}{\mu (\frac 12 + i \beta)} \in K''$$

			So, to get the minimum of $|F|$ on $K''$,  we only have to consider the two functions $\psi_+$ and $\psi_-$ of a single real variable $\beta$ defined by $\psi_\pm (\beta) = |F(y_\pm (\beta))| = \left| \frac 12 + i \beta \right| |y_\pm (\beta ) |$. A straightforward but tedious computation proves that the minimum of both these functions occurs at $\beta=0$ and is equal to
			
			$$\psi_\pm (0) = \left\vert \dfrac{1 \pm \sqrt{1 - \frac{\mu}{4} } }{\mu} \right\vert$$
			
		Finally, using explicit formulas for $\tau_\mu$ as a solution of a cubic equation, one can infer that $\psi_\pm (0)  \geq |\tau_\mu|^2$.
			
		Note that this last inequality is an equality if and only if $\mu = 4$. And as we assumed that $\mu \neq 4$, it implies that the minimum of $|F|$ is only attained on the boundary defined by $|y| = |\tau_\mu|$, and hence $|q| = |\tau_\mu|$.
		
		\item Step 4 : We conclude that $p=q=r = \tau_\mu$.
		
			 We already proved that if $|y| = |\tau_\mu|$, then $|\zeta(y,y)| \geq |\tau_\mu|$, and hence $|F(y)| \geq |\tau_\mu|^2$.  So we see that the minimum of $|F|$ is attained when $y = \tau_\mu$. 
		
		So this means that $q = r = \tau_\mu$, which implies that $p = \tau_\mu$, and conlcudes the proof of the Lemma.

\end{enumerate}
\end{proof}

		\subsubsection{Upper bound for the minimum of a Markoff map}

%		There are Markoff maps such that no vertex is a sink. In that case, one can follow an infinite path on oriented edge, and in that case we can use the following proposition (\cite{tan_gen}, Lemma 3.11), 
%		\begin{Proposition}\label{prop:escaping} Let $\phi$ be a $(0,0,0,\mu)$-Markoff map with $\mu \neq 4$, and assume that $\beta$ is an infinite ray in $\mathcal{T}$ consisting of a sequence $(e_n)_{n\in \N}$ of edges of $\mathcal{T}$ such that the arrow on each $e_n$ assigned by $\phi$ is directed towards $e_{n+1}$. Then there exists at least one region $X \in \Omega$ with $|\phi (X) | <2$.
%		\end{Proposition}
%		
%		Combining this proposition with Theorem \ref{thm:sink1} we obtain :

The explicit value of the sink constant gives rise to a more general statement on the minimum value for any Markoff map.  
		
		\begin{Theorem}\label{thm:region} Let $\mu \neq 4$ and $\phi$ be a $(0,0,0,\mu)$-Markoff map. Then there exists  $X \in \Omega$ such that $|\phi (X) | \leq |t_\mu|$.
		\end{Theorem}

		\begin{proof}
			Start from any vertex $v_0 \in \mathcal{T}$. If this vertex is not a sink, there exist another vertex $v_1$ adjacent to $v_0$ such that the edge from $v_0$ to $v_1$ is oriented towards $v_1$ with the orientation given by $\phi$. Continuing this process allows us to construct a sequence $(v_n)$ of vertices with the property that $v_{n+1}$ is a vertex adjacent to $v_n$ and the orientation given by $\phi$ of the edge from $v_n$ to $v_{n+1}$ is oriented towards $v_{n+1}$. We assume that we have a maximal sequence, so there are only two behaviors that can occur :
			\begin{itemize}
				\item The sequence terminates at some vertex, and cannot be continued. In that case, this last vertex is necessarily a sink. So using Theorem \ref{thm:sink1}, we have that one of the regions $X$ around that terminal sink is such that $|\phi (X) | \leq |t_\mu|$
				\item The sequence is infinite, in which case we have a so-called \emph{escaping ray}, and hence we can apply the result (see \cite{tan_gen}, Lemma 3.11) stating that along such an infinite oriented path, there exists at least one region $X \in \Omega$ such that $|\phi (X) | <2$, and hence $|\phi (X)| \leq |t_\mu |$.
					
			\end{itemize}		
		\end{proof}

	\begin{Remark}
	
		In the case $\mu = 4$, the Theorem no longer applies. Indeed, $t_4 = 2$,  but there are $(0,0,0,4)$-Markoff maps such that for all regions $X$ we have $|\phi (X) | > 2$. 
		
		Note that if $\phi$ is such a Markoff map, then no vertex is a sink for $\phi$, and we have $\inf \{ |\phi (X) | , X \in \Omega \}  = 2 = t_4$.
		
	\end{Remark}

		\subsubsection{Real Case}

When $\Um \in \R^4$, one can consider real $\Um$-Markoff maps $\phi : \Omega  \longrightarrow \R$, and we denote by ${\bf \Phi}_{\Um}^\R$ the set of such maps. 

When $\mu \in \R$, we have precise information on the real roots of the polynomial equation $X^3 - 3X^2 + \mu = 0$. The discriminant of such an equation is 
$$\Delta  = 108\mu - 27 \mu^2 = 27 \mu (4-\mu)$$
%= 18abcd - 4b^3d+b^2c^3-4ac^3-27a^2d^2
So we can easily distinguish 5 possible cases depending on the value of $\mu$.

\begin{Proposition}\label{prop:realroots} The real roots of $X^3 - 3X^2 + \mu = 0$ satisfy the following properties :
\begin{enumerate}
	\item If $\mu < 0$, then the equation has a unique real solution. This solution is positive and greater than $3$.
	\item If $\mu = 0$, then there are two real solutions : $0$ and $3$.
	\item If $\mu \in ]0,4[$, then $(E_\mu)$ has three real solutions. Exactly one of them is in $]2,3[$ and the other two are in $]-2, 2[$.
	\item If $\mu = 4$, then there are two real solutions : $-1$ and $2$.
	\item If $\mu > 4$, then this equation has a unique real solution, which is negative. In that case $|t'_\mu| < 2$ if and only if $\mu < 20$.
		
\end{enumerate}
\end{Proposition}

Now, one can refine the Theorem \ref{thm:sink1} in the case $\Um = (0,0,0,\mu)$ with $\mu \in \R$ and for real Markoff maps. %as only the real roots of the polynomial $X^3 - 3X^2 + \mu = 0$ will play a role.

\begin{Theorem}\label{thm:sinkreal} Let $\mu \in \R \setminus \{ 4 \}$. Let $t'_\mu$ be the largest (in absolute value) real root of the polynomial $X^3 - 3X^2 +\mu = 0$. Then for all $\phi \in {\bf \Phi}_{\Um}^\R$, if three regions $(X,Y,Z)$ meet at a sink, then :
$$\min \{ |x| , |y| , |z| \} \leq \max ( |t'_\mu|, 2)  $$
\end{Theorem}

\begin{proof} We use a different strategy for this Theorem and take advantage of the fact that we are dealing with real functions.

We consider the set

$$\scc = \left\{ (x,y,z) \in (\R\setminus \{0\} )^3 \ | \  \frac{x}{yz}\leq\frac 12 , \frac{y}{xz} \leq \frac 12, \frac{z}{xy} \leq \frac 12 \right\}$$

This set corresponds to triple of non-zero real numbers such that if a Markoff map has values $x,y,z$ around a vertex, then this vertex is a sink. Hence we will call $\scc$ the \emph{sink domain}.

First, we can see that for all $\mu \in \R$, the triple $(t'_\mu , t'_\mu , t'_\mu) \in \scc$. Indeed, if $\mu >4$, then $t'_\mu <0$ is negative and all the equations are trivially satisfied because $\frac{1}{t'_\mu} <0 < \frac 12$. On the other hand, if $\mu \leq 4$, then $t'_\mu \geq 2$ and hence $\frac{1}{t'_\mu} \leq \frac 12$. We can also see trivially that $(2,2,2) \in \scc$.

Let $f (x,y,z) = x^2 + y^2 + z^2 - xyz$. We compute the gradient of the function $f$, so we get 
$$\nabla f = \begin{pmatrix} 2x - yz \\ 2y -xz \\ 2z - xy \end{pmatrix}$$
Note that if $(x,y,z) \in \scc$ and $x,y,z$ are all of the same sign, then each coordinates of $\nabla f$ is negative at $(x,y,z)$. We have $\nabla f = (0,0,0)$ if and only if $x=y=z=2$.  

Now, let $(x,y,z) \in \scc $ such that $f(x,y,z) = \mu$, assume by contradiction that $|x|, |y|, |z| > \max (|t'_\mu|, 2)$. 

As $f(x,y,z) = f(-x, -y, z) = f(-x, y, -z)  = f(x, -y, -z)$, we can assume without loss of generality that $x,y,z$ are all of the same sign.

\begin{itemize}
	\item If $xyz < 0$, then we have  $\mu = f(x,y,z) = x^2 + y^2 + z^2 - xyz \geq 20 >4 $. This implies that $t_\mu < 0$, and hence $(t'_\mu, t'_\mu, t'_\mu)$ and $(x,y,z)$ are both in $]- \infty , 0 [^3$. Consider the straight path $\nu : [0,1] \rightarrow \R^3$ from such that $\nu(0) = (t'_\mu, t'_\mu, t'_\mu)$ and $\nu (1) = (x,y,z)$. This path is entirely contained in $\scc$ and hence $f\circ \nu$ is a strictly increasing function. This means that $f(x,y,z) > f(t'_\mu, t'_\mu, t'_\mu)  = \mu$, which gives a contradiction.
	\item If $xyz >0$, assume without loss of generality that $2<x\leq y\leq z$ then consider the jagged path $\nu: [0,1] \rightarrow \R^3$ defined by three straight path :
	$$(2,2,2) \rightarrow (x,x,x) \rightarrow (x,y,y) \rightarrow (x,y,z)$$
	This path is strictly increasing in each variable and stays in $\scc$. Hence $f \circ \nu$ is a strictly decreasing function. So $f(x,y,z) < f(2,2,2)=4$, which implies that $\mu < 4$ and $t'_\mu > 2$.
	
	Now, that same path $\nu$ also joins $(t'_\mu , t'_\mu, t'_\mu)$ to $  \rightarrow (x,y,z) $ . By the same argument, we have $f(x,y,z) < f(t'_\mu ,t'_\mu , t'_\mu) = \mu$, which gives a contradiction.
\end{itemize}

\end{proof}

\begin{Remark}
When $\mu > 20$ we have $|t'_\mu| < |t_\mu|$ and hence this theorem is stronger than the previous one in the case of real Markoff maps. For an explicit example, consider the case $\mu = 54$. The only real root of the equation $X^3-3X^2+54 = 0$ is $t'_\mu = -3$, however, a dominant complex root is given by $t_\mu = 3+3i$, and hence $|t'_\mu| < |t_\mu|$. 
\end{Remark}

	\subsection{General case}\label{sub:general_case}
	
	We now consider the general case of $\Um  \in \C^4$, where the previous arguments cannot work as we can see in the following example.

\begin{Example}
Let $\Um= (-50, 30, 50, 0)$. In that case, the triple $(-10, -10, 10)$ is a $\Um$--Markoff triple, and it's easy to check that this triple corresponds to a sink. However the largest root of 
$$X^3 - 3X^2 - (-50+30+50)X = 0$$
is $\frac 12 (3+\sqrt{129} ) \approx 7.17... < 10$.

Moreover, the $\Um$-Markoff triple $(t_\mu, t_\mu, t_\mu)$ is not a sink. 

\end{Example}
	
	This suggests that the naive generalization of previous results is false, and the optimal constant in the general case could be much more difficult to obtain. But nonetheless, we can adapt the proof of Theorem \ref{thm:sinkreal} and see that there are certain cases of geometrical significance that can be studied with similar methods. In particular, we will consider the following set of parameters :

$$U = [0, +\infty [^3 \times ]-\infty , 4] \subset \R^4$$
	
	\begin{Definition}
		Let $\Um \in U$. A Markoff map $\phi \in {\bf \Phi}_{\Um}$ is said to be \emph{positive} if $\rm{Im} (\phi) \subset \R_+$. We denote by ${\bf \Phi}_{\Um}^{\geq 0}$ the set of positive Markoff maps.
	\end{Definition}

	\begin{Theorem}\label{thm:sink4pos}
	Let $\Um=(\lambda_1,\lambda_2, \lambda_3, s)$ in $U$. For all $\phi \in {\bf \Phi}_{\Um}^{\geq 0}$, if three regions $X_1,X_2,X_3$ meet at a sink, then 
	$$\min \{ x_1 , x_2 , x_3 \} \leq T_{\Um} $$
	where $T_{\Um}$ is the largest positive real root of $X^3-3X^2 - (\lambda_1+\lambda_2+\lambda_3)X + s = 0$.
\end{Theorem}

	\begin{proof}
	The proof is similar to the proof of Theorem \ref{thm:sinkreal}.
	
	Let $\phi \in {\bf \Phi}_{\Um}^{\geq 0}$ and let $X_1,X_2,X_3 \in \Omega$ be three regions meeting at a sink. The sink condition can be written :
$$|x_i| \leq |x_j x_k-x_i- \lambda_i| , \ \mbox{ with }  \{i,j,k \} = \{1, 2, 3 \} $$

As the Markoff map is positive, we have that both $x_i$ and $x_j x_k-x_i-\lambda_i$ are positive and hence the condition is equivalent to $x_j x_k -  2x_i- \lambda_i \geq 0$. This leads to the following definition of the \emph{Sink domain} as before : 
$$\scc' = \big\{ (x_i, x_j, x_k) \in [2, +\infty[ \, \big| \, x_j x_k -2x_i- \lambda_i \geq 0 , \mbox{ for all }  \{i,j,k \} = \{1, 2, 3 \} \big\}$$

Let $t = T_\mu$ be the largest real root of $ P(X) = X^3-3X^2 - (\lambda_1+\lambda_2+  \lambda_3)X + s $. As $P(2) = s-4 -2(\lambda_1+\lambda_2+\lambda_3) <0$, we know that $t> 2$. Moreover, we see that $(t, t, t) \in \scc'$ as $t(t^2-2t-p) = t^2 + (q+r)t - s >0$.

Now, consider the function $f(x,y,z) = x^2+y^2+z^2-xyz+\lambda_1 x+ \lambda_2 y+ \lambda_3 z$. The gradient of this map is 

$$\nabla f = \begin{pmatrix} 2x-yz+\lambda_1 \\ 2y-xz+\lambda_2 \\ 2z-xy+\lambda_3 \end{pmatrix}$$

So on $\scc'$, the coordinates of the gradient are negative. So if $\min (x_1,x_2,x_3) > t$, then as before we can consider a path $\nu$ from $(t,t,t)$ to $(x_1,x_2,x_3)$ that stays in $\scc'$ and is increasing in each variable. This implies that $f \circ \nu$ is striclty decreasing and hence $f(x_1,x_2,x_3) < f(t,t,t)=s$ which gives a contradiction.
\end{proof}

	 \section{Character Varieties of surface groups}\label{s:char}
	 
	 In this section, we recall the precise relationship between Markoff maps and representations of the fundamental group of the one holed torus, the four-holed sphere and the non-orientable surface of genus $3$. For more details, we refer to \cite{tan_gen, mal_ont}.
	 
	 Note that a natural generalisation of Markoff maps, called markoff quads (\cite{mal_ont2,hua_sim}) can be related to representations of a three-holed projective plane. We will not discuss this case in this article but we expect similar methods to work for finding trace systoles of these representations.

	\subsection{One-holed Torus}
		 
		 Let $T$ be a topological one-holed torus, and $\Gamma$ be its fundamental group. The group $\Gamma = \langle \a , \b \rangle$ is the free group of rank two, and $\a$ and $\b$ correspond to simple closed curves with geometric intersection number one. 

The set $\widehat{ \Omega}$ of free homotopy classes of essential unoriented simple closed curves on $T$ can be naturally identified with $\Q \cup \{ \infty \}$ by considering the ``slope''  of the curve (see \cite{tan_gen}). It is a well-known fact that the curve complex of the one-holed torus is isomorphic to the Farey triangulation $\mathcal{F}$ described in Section \ref{s:mar}. This means that we can naturally identify $\widehat{\Omega}$ with $\Omega$, so that each region in $\Omega$ correspond to a simple closed curve. As a consequence, two regions in $\Omega$ share an edge if and only if the corresponding curves intersect exactly once on $T$. Similarly, if three regions meet at a vertex, then there exists a generating set $(u , v )$ of $\Gamma$ such that the corresponding curves have representatves $u, v , uv \in \Gamma$.

For $k \in \C$, a representation $\rho : \Gamma \rightarrow \mathrm{SL} (2 , \C)$ is said to be a $k$-representation, if $\tr ([\a,\b]) = k$, where $[\a , \b ] = \a \b \a^{-1} \b^{-1}$ is the element corresponding to the boundary curve of $T$. Note that this element is independent of the chosen basis for $\Gamma$. The space of equivalence classes of $k$-representations is denoted $\mathcal{X}_k$ and is called the $k$-relative character variety. 

There is a natural one-to-one correspondance between $\mathcal{X}_k$ and ${\bf \Phi}_{\Um}$ with $\Um =(0,0,0, k+2)$ obtained by fixing a generating set $\a , \b$ for $\Gamma$. Indeed, from this generating set, one can identify $\Omega$ and $\widehat{\Omega}$, and hence a character $\rho : \Gamma \rightarrow \mathrm{SL} (2 , \C)$ gives rise to a map $\phi : \Omega \rightarrow \C$ defined by $\phi (X) = \tr (\rho (g))$ where $g \in \Gamma$ is the representative of the element in $ \widehat{\Omega}$ corresponding to $X$.

The  edge and vertex relations then follow from the classical trace identities in $\mathrm{SL} (2 , \C)$ :
	\begin{eqnarray}
		\tr A \tr B  =& & \tr AB + \tr AB^{-1} ,\\
		(\tr A)^2 + (\tr B)^2 + (\tr AB)^2 - \tr A \tr B \tr AB   =& &  \tr [A,B] +2 = k +2
	\end{eqnarray}

Conversely, any $\Um$--Markoff map $\phi$ gives rise to an equivalence class of representation in $\mathcal{X}_{\mu - 2}$, once given a choice of three adjacent regions $(X, Y, Z)$. Indeed, if we consider the $\mu$ Markoff triple $(x,y,z) = (\phi( X) , \phi ( Y) , \phi (Z))$. We know that $\mathcal{X}_{\mu -2 }$ is identified with 
$$\{ (x,y,z) \in \C^3 \ | \ x^2 + y^2 + z^2 - xyz -2 = \mu -2 \},$$
and hence the triple $(x,y,z)$ defines a unique $k$-character in $\mathcal{X}_k$.

	\subsection{Four-holed Sphere}
	
	Let $S$ be a topological four-holed sphere, and $\Gamma$ be its fundamental group. The group $\Gamma$ admits the following standard presentation
	$$\Gamma = \langle \a , \b , \g , \d \mid \a \b \g \d \rangle$$
	where $\a, \b , \g , \d$ correspond to homotopy classes of the four boundary components. Note that $\Gamma$ is isomorphic to the free group on three generators $\langle \a , \b, \g \rangle$. 
	
	As in the one-holed torus case, the set $\widehat{ \Omega}$ of free homotopy classes of essential unoriented simple closed curves on $T$ can be naturally identified with $\Q \cup \{ \infty \}$ by considering the ``slope''  of the curve, see \cite{mal_ont}.
	
	 For $\Ut = (a,b,c,d) \in \C^4$, a representation $\rho : \Gamma \rightarrow \mathrm{SL} (2 , \C)$ is said to be a $\Ut$-representation, when $\tr (\rho(\a)) = a$, $\tr (\rho (\b )) = b$, $\tr ( \rho ( \g)) = c $ and $\tr (\rho (\d)) = d$.
	  
	 The space of equivalence classes of $\Ut$-representation is denoted $\X_{\Ut}$ and is called the $\Ut$-relative character variety. We can consider the map
	 \begin{align*}
	 	\chi :  \X_{\Ut} & \longrightarrow \C^3 \\ [\rho] & \longmapsto \left( \tr (\rho (\a \b)) , \tr (\rho (\b \g)), \tr (\rho (\g \a)) \right)
	\end{align*}
	A classical result (see Goldman \cite{gol_tra}) on the character variety of the free group in three generators states that this map is injective and its image is given by :
	$$\X_{\Ut} = \{ (x,y,z) \in \C^3 \ | \ x^2 + y^2 + z^2 + xyz = px + qy + rz + s \} $$
	
	with $(p,q,r,s)= GT (a,b,c,d)$ where $GT : \C^4 \rightarrow \C^4$ is the map defined by :
	 
		$$\begin{pmatrix} a \\ b \\ c \\ d \end{pmatrix}  \longmapsto \begin{pmatrix} ab+cd \\ ad+bc \\ ac+bd \\ 4 - a^2 - b^2 - c^2 - d^2 - abcd \end{pmatrix} = \begin{pmatrix} p \\ q \\ r \\ s \end{pmatrix}$$
	 This map was studied by Goldman-Toledo \cite{gol_aff} who proved that it is onto and proper (see also \cite{cantat_dyn}). 
	 
	 This allows us to identify the relative character variety with the space of Markoff maps as follows
	 
	 \begin{Proposition}
	  Let $\Ut \in \C^4$.  A representation $\rho$ is in $\X_{\Ut}$ if and only if the triple $-\chi (\rho)$ is a $\Um$-Markoff triple with $\Um = GT (\Ut)$. 
	  \end{Proposition}
	  
	  Note that there is a sign convention that is slightly different from previous work in \cite{mal_ont}, and this is due to the fact that we give a general definition that works for both the one-holed torus and the four-holed sphere. 
	 
	 \subsection{Closed surface of characteristic $-1$}\
	 
	 Let $N_3$ be a topological closed surface of characteristic $-1$. It is the non-orientable surface of genus $3$, namely the connected sum of three projective plane.
	 
	 \subsubsection{Curves on $N_3$}
	 
	 Recall that a closed curve on a non-orientable surface is said to be \emph{two-sided} if it admits a regular neighborhood which is orientable, else it is said to be \emph{one-sided}. We also say that a simple closed curve is \emph{orientable} (resp. \emph{non-orientable}) if the surface cut along that curve is orientable (resp. non-orientable). Note that on a non-orientable surface, a separating curve is necessarily $2$-sided and non-orientable. And if the genus of the non-orientable surface is odd, there are no orientable $2$-sided curves. 
	 
	 So, on the surface $N_3$, there are exactly four types of simple closed curves :
	\begin{enumerate}
		\item Separating curves. These are necessarily $2$-sided and non-essential, as they bound a Möbius band. 
		\item Non-separating $2$-sided curves. The surface obtained by cutting along such a curve is a 2-holed projective plane.
		\item Orientable $1$-sided curve. There is a unique such curve
		\item Non-orientable $1$-sided curves. A curve such that $S \setminus \gamma$ is non-orientable. 
	\end{enumerate}
	
	 The curve complex of $N_3$ has the following structure :
	 \begin{itemize}
	 	\item The unique orientable 1-sided curve is disjoint from all non-separating $2$-sided curves, but  intersect all non-orientable 1-sided curve.
	 	\item Two $2$-sided non-separating curves intersect at least once. And for each $2$-sided non-separating curve, there is a unique non-orientable curve that is disjoint from it. 
	 	\item Finally, the subcomplex formed by non-orientable $1$-sided curves is equivalent to the curve complex of the one-holed torus.
	\end{itemize} 
	
	 \subsubsection{Character variety of $N_3$}

	 The fundamental group of $N_3$ is given by the following presentation :
$$\Gamma_3 = \pi_1 (N_3) = \langle \a , \b , \g \mid \a^2 \b^2 \g^2 \rangle$$
where $\a, \b , \g$ are homotopy classes of disjoint simple non-orientable $1$-sided curves. The unique orientable $1$-sided curve is given by $\delta = \a \b \g$. 

The character variety of $\Gamma_3$ in $\mathrm{SL} (2, \C)$ is given by the following Theorem, which already appeared in the author's thesis (\cite{pal_these}) :

\begin{Theorem} The map
	\begin{align*}
		\X (\Gamma_3) & \longrightarrow \C^4 \\
		[ \rho ] & \longmapsto (\tr (\rho (\a)), \tr (\rho ( \b)), \tr (\rho (\g)) , \tr (\rho ( \d)) )
	\end{align*}
	is injective. Its image is the set
	$$\mathfrak{N} = \left\{ (a,b,c,d) \in \C^4 \, \big| \, a^2 + b^2 + c^2 - abc \frac d2 = 4 \right\}$$
\end{Theorem}

\begin{proof} 
We know that $\X (\Gamma_3)$ is an algebraic subset of the character variety of the free group in three generators. Which means that we have an injective map $\X (\Gamma_3) \rightarrow \C^7$ whose image is the set 
$$
	\left\{  (a,b,c,d,x,y,z) \in \C^7 ,  \left| \begin{array}{l}a^2+b^2+c^2+d^2+x^2+y^2+z^2-abcd-4 \\ -(ab+cd)x-(ad+bc)y-(ac+bd)z +xyz =0  \end{array} \right.  \right\}$$

The relation in $\Gamma_3$ implies that $(\a \b)^{-1} = \b \g \g \a$ and hence any $(a,b,c,d,x,y,z) \in \X (N_3)$ satisfies $x = cd - x$. Similarly, we have $y = ad-y$ and $z = bd - z$. Equivalently, we can write $x = \frac{cd}{2} , y = \frac{ad}{2} , z = \frac{bd}{2}$, so we substitute the expression of $x,y,z$ in terms of $a,b,c,d$ in the equation defining $\X (F_3)$ we get :
\begin{align*}
& a^2+b^2+c^2+d^2+ \left(\frac{cd}{2}\right)^2 + \left(\frac{ad}{2}\right)^2 + \left(\frac{bd}{2}\right)^2 -abcd - 4 \\
& - (ab+cd)\left(\frac{cd}{2}\right) - (ad+bc)\left(\frac{ad}{2}\right) - (ac+bd)\left(\frac{bd}{2}\right)  + \left(\frac{cd}{2}\right)\left(\frac{ad}{2}\right)\left(\frac{bd}{2}\right)  = 0\\
\end{align*}
which is equivalent to $ a^2 + b^2 + c^2 - abc \frac d2 = 4$

This proves the injectivity of the map described in the theorem, and that its image is included in $\mathfrak{N}$. To get surjectivity, we can use Theorem 3.2 in \cite{gonzalez_ont} which describe $X(G)$ as an explicit closed algebraic set (see the author thesis for more details).
\end{proof}

 This allows us to parametrize directly $\X (\Gamma_3)$ as a set of Markoff maps, except in the $d = 0$ case.
 %If $d=0$, then the equation defining the character variety becomes $a^2+b^2+c^2=4$.

\begin{Proposition}

An element $(a,b,c,d) \in \C^4$ with $d \neq 0$, is the character of a representation $\rho$ in $\X (\Gamma_3)$ if and only if $\left( \frac{cd}{2}, \frac{ad}{2},\frac{bd}{2} \right)$ is a $(0,0,0,d^2)$-Markoff map. 

An element $(a,b,c,0)$ is the character of a representation $\rho$ in $\X (\Gamma_3)$ if and only if $a^2+b^2+c^2=4$.
\end{Proposition}

\begin{proof}

When $d\neq 0$, the equation defining $\mathfrak{N}$ is equivalent to :
$$x^2 + y^2 + z^2 - xyz = d^2$$
with the change of variable $x = \frac{cd}{2} , y = \frac{ad}{2} , z = \frac{bd}{2}$.
 \end{proof}

%%%%%%%%%%%%%%%%%%%%%%%%%%%%%%%%%%%%%%%%%%%%%%%%%%%%%%%	 
	 \section{Trace Systoles} \label{s:sys}
	 
	 The goal of this section is to relate the sink constant of Section \ref{s:sink} with the trace systoles of representations of surface groups, and systoles of certain hyperbolic manifolds. This will allow us to prove Theorems \ref{introthm:sys1} and \ref{introthm:sys2}.

	 \subsection{From hyperbolic surfaces to representations}\

	 Let $S$ be an hyperbolic surface of finite type that can be closed or with geodesic boundaries, cusps and conical singularities, so that $\chi (S) <0 $. 
	 
	 Recall that if $S$ is an orientable surface of genus $g$, then $\chi (S) = 2 - 2g -b -s + \sum (\frac{\alpha_i}{2\pi} - 1 )$ where $b$ the number of geodesic boundaries, $s$ the number of cusps and $\alpha_i$ the angles of the conical singularities. When $S$ is a non-orientable surface of genus $k \geq 1$, we have $\chi(S) =2 - k -b -s+ \sum (\frac{\alpha_i}{2\pi} - 1 )$.  We denote by $\mathfrak{B}$ the so-called \emph{boundary data}, namely the number of cusps, the lengths of boundary components and the angles of conical singularities, if any.

	 The \emph{systole} of $S$ is the minimal length of an essential simple closed curve on $S$ and is denoted $\mathrm{sys} (S)$. This defines a function $\mathrm{sys} : \mathscr{T} (\Sigma, \mathfrak{B}) \rightarrow \R_+$ where $\mathscr{T}(\Sigma, \mathfrak{B})$ is the Teichmüller space of equivalence classes of hyperbolic structures on the topological surface $\Sigma$ with prescribed boundary data $\mathfrak{B}$.

The holonomy representation of such a structure gives rise to an homomorphism from $\pi_1 (\Sigma)$ into $\mathrm{PGL} (2 , \R)$, where $\Sigma$ is the topological surface corresponding to $S$. For consistency, we consider that conical singularities do not belong to $S$, so that a closed curve around such a singularity is a non-trivial element of the fundamental group (but is not an essential curve). There is a relation between the length of a closed geodesic on $S$ and the trace of its image by the holonomy representation. 

Note that in the orientable case, all curves are $2$-sided and the representation $\rho_S$ takes values in $\mathrm{PSL} (2 , \R)$, while in the non-orientable case, we have to distinguish whether the curve is $2$-sided or $1$-sided, and we know that the holonomy representation of an hyperbolic structure sends every $1$-sided curve to an element of $\mathrm{PGL}^- (2 , \R)$, the subset of orientation-reversing isometries of the hyperbolic plane. So we have

 	\begin{equation}\label{eqn:sys-tys}
 	l_S (\gamma) = \left\{ \begin{array}{ll} 2 \mathrm{arccosh} \left( \frac{| \tr (\rho_S (\gamma)) |}{2} \right) & \mbox{ if } \gamma \mbox{ is 2-sided} \\
				2 \mathrm{arcsinh}  \left( \frac{| \tr (\rho_S (\gamma)) |}{2} \right) & \mbox{ if } \gamma \mbox{ is 1-sided}\end{array} \right.
	 	\end{equation}

	 From these relation, we see that we can generalize the notion of systole to the entire space of representations $\mathrm{Hom} (\pi_1 (\Sigma), \rm{SL} (2 , \C))$. 
	 
	 \begin{Definition}Given a representation $\rho : \pi_1 (\Sigma) \rightarrow \rm{SL} (2  , \C)$, we define the trace systole of $\rho$ as :
	
	 	$$\rm{tys} (\rho) = \inf \left\{ |\tr (\rho (\gamma)) | \ , \ \gamma \in \pi_1 (\Sigma), \mbox{essential simple closed curve} \right\}$$
	 	\end{Definition}
	 
	  The definition of $\rm{tys}$ can be used for any representation into $G = \rm{PSL} (2 , \R)$ or  $\rm{PSL} (2 , \C)$, because the modulus $|\tr (\rho (\gamma))|$ is still well-defined in each of these groups. So we will use the same notation for such representations. We also note that as the trace is conjugation invariant, the function $\rm{tys}$ can be defined on the character variety $\mathfrak{X} (\pi_1 (\Sigma) , G ) = \rm{Hom} (\pi_1 (\Sigma) , G) / G$ which is the space of orbit closures for the action of $G$ by conjugation on representations.
	 
	 It is important that we restrict ourselves to simple curves in the definition of the trace systole to have an interesting function for representations that are not discrete. Indeed, if $\rho : \pi_1 (\Sigma) \rightarrow \rm{PSL} (2 , \R)$  is a representation that is not discrete,  then its image is dense in $\rm{PSL} (2 , \R)$ and hence $\inf \left\{ |\tr (\rho (\gamma)) | , \gamma \in \pi_1 (\Sigma) \right\} = 0$.  But the trace systole of such a representation is not necessarily $0$,  for example if $\rho$ is the holonomy of an hyperbolic structure on a torus with a conical singularity of irrational angle, then the representation is dense but the hyperbolic systole is well-defined and non-zero.

	\subsection{Maximum of the trace systole}\

	When $\Sigma$ is a closed surface, the function $\rm{tys}$ is bounded on the character variety and attains its maximum, that we denote by $\rm{Tys} (\Sigma)$ this maximum. When $\Sigma$ has boundary components, the fundamental group $\pi_1 (\Sigma)$ is a free group and the function $\rm{tys}$ is unbounded on the whole character variety. However, one can consider its restriction on relative character variety with prescribed traces on the boundary.
	
	If $\Sigma$ has $p$ boundary components, with $c_1 , \dots , c_p \in \pi_1 (\Sigma)$ representing loops around each boundary, and we let $\mathfrak{B} = (b_1 , \dots , b_p) \in \C^b$, we can consider 
	$$\mathfrak{X}_{\mathfrak{B}} (\Sigma) = \left\{ [\rho] \in \mathfrak{X} (\Sigma) , \forall i \in \{ 1 , \dots , p \} , \tr (\rho (c_i) = b_i \right\} $$
	
	So we can define $\mathrm{Tys}_{\mathfrak{B}} (\Sigma)$ as the maximum of $\rm{tys}$ on $\mathfrak{X}_{\mathfrak{B}} (\Sigma)$.

	 	\subsection{One holed torus}\

	A reformulation of Theorem \ref{thm:region} in terms of trace systole directly gives Theorem \ref{introthm:sys1}.(1)
	
	\begin{Theorem}\label{thm:tysT} Let $T$ be a one-holed torus, and $k \in \C \setminus \{ 2 \}$. We have 
	$$ \mathrm{Tys}_{k} (T) = |t_{k+2}|$$
	
	\end{Theorem}

	Recall that $t_a \in \C$ is the dominant root of $X^3 - 3X^2+a = 0$.

	\subsubsection{Hyperbolic one-holed torus}

	In the case of representations coming from holonomy representations of hyperbolic structures, we can use the previous Theorem to infer optimal systolic inequalities given in Theorem \ref{introthm:sys2}.(1),(2) and (3).

\begin{Theorem}\label{thm:sysT11} Let $T$ be a surface with an hyperbolic metric and $\rm sys$ the length of its systole.
	
	\begin{enumerate}
		\item If $T$ is an hyperbolic one-holed torus, with geodesic boundary of lenth $l$, then 
			$$\cosh \left( \frac{{\rm sys}}{2}\right)  \leq \cosh \left( \frac l6 \right) + \frac 12$$
		\item If $T$ is a once-punctured torus, then
			$$\cosh \left( \frac{{\rm sys}}{2}\right)  \leq  \frac 32$$
		\item If $T$ is a singular hyperbolic structure on the torus with a conical singularity of angle $\theta \in [0 , 2\pi [$, then
			$$\cosh \left( \frac{{\rm sys}}{2}\right) \leq \cos \left( \frac \theta6 \right) + \frac 12$$
	\end{enumerate}
\end{Theorem}

%\begin{Remark}
%The first two systolic inequalities were well-known before (see \cite{schmutz-systole, gen_pay}). The last inequality seems to be new, at least in this form, but might have been known to the specialists. In any case, the proof that we propose is quite different than the one in the litterature.
%\end{Remark}

\begin{proof}
	Let $\rho : \pi_1 (T) \rightarrow \rm{PSL} (2, \R)$ be the holonomy representation of the hyperbolic structure on $T$.  Using Theorem \ref{thm:tysT} we have that $\rm{tys}(\rho) \leq |t_\mu|$ with $3 t_\mu^2 -  t_\mu^3 =  \mu $.
	
	On the other hand, if we denote $\pi_1 (T) = \langle \a, \b \rangle$, then the element $[\a,\b]$ corresponds to the boundary component of $T$ (or the conical singularity). Let $A = \rho (\a)$ and $B = \rho(\b)$, so we have $\tr ([ A,B] ) = \mu - 2$. Depending on the case we have :
	
	$$\tr ([ A, B])  = \left\{ \begin{array}{ll} - 2 \cosh \left( \frac l2 \right) & \mbox{ if } T \mbox{ is an hyperbolic one-holed torus} \\
						-2 & \mbox{ if } T \mbox{ is once punctured torus} \\
						- 2 \cos \left( \frac \theta2 \right) & \mbox{ if } T \mbox{ is a singular torus} \end{array}
						\right.
						$$

	In the first case, using the trigonometric identity $\cosh (3x) = 4 \cosh^3 (x) - 3 \cosh(x)$, and setting $x = \frac l6$ we see that the equation can be rewritten  
	$$\begin{array}{lrl}
		 &   t_\mu^3 - 3 t_\mu^2  +2 & =  2 \cosh (3x) \\
		\Leftrightarrow & 4 \left(\dfrac{t_\mu}{2} - \dfrac 12 \right)^3 - 3 \left( \dfrac{t_\mu}{2} - \dfrac 12 \right) & = \cosh (3x) \\
		  \Leftrightarrow & \dfrac{t_\mu}{2} - \dfrac 12 & = \cosh (x)
	\end{array}$$
	
	We know that $\rm{tys} (\rho) = 2 \cosh \left( \frac{\rm{sys} (T) }{2} \right)$. Hence we have that $\cosh \left( \frac{\rm{sys} (T) }{2} \right) \leq \cosh (x) + \frac 12$.

	The exact same computations apply in the third case, replacing $\cosh (3x)$ by $\cos (3x)$.  The second case simply corresponds to $\mu = 0$ and is obtained directly. 	

\end{proof}

		\subsubsection{Non-Fuchsian component for the one-holed torus}

We also get original results for representations of the fundamental group of the one-holed torus that do not correspond to hyperbolic structure (singular or not) on a one-holed torus. These corresponds to representations with relative Euler class $0$.

\begin{Theorem}\label{thm:sysT11NF} Let $\rho \in \mathcal{X}_k (T)$, with $k > 2$. 
	\begin{enumerate}
		\item If $k \in ]2 , 18[$, there exists a simple closed curve that is sent to an elliptic element.
		\item If $k \geq 18$, there exists a simple closed curve $\gamma \in \pi_1 (T)$ with $|\tr (\rho (\gamma)| \leq 2 \cosh \left( \frac{l}{6} \right) -1$, with $l = 2 \cosh^{-1} \left( \frac k2 \right)$.
	\end{enumerate}
\end{Theorem}

\begin{proof}A representation in $\mathcal{X}_k (T)$ with $k>2$, corresponds to a $\mu$-Markoff map with $\mu >4$. Using the same reasoning as in Theorem \ref{thm:region}, we deduce that this Markoff map either has a sink or an infinite descending ray. If there is an infinite descending ray, then there exists a curve $\gamma$ such that $|\tr (\rho (\gamma)) | < 2$.

Otherwise, we use  Theorem \ref{thm:sinkreal} and Proposition \ref{prop:realroots} to get the following :
	\begin{itemize}
		\item If $k \in ]2 , 18[$, then $\mu = k+2 \in ]4 , 20 [$, and hence there exists a simple closed curve $\gamma$ such that $|\tr (\rho (\gamma))| \leq |t_\mu| < 2$. So this curve $\gamma$ is sent to an elliptic element.
		\item If $k \geq 18$, then there exists a curve $\gamma$ such that $|\tr (\rho (\gamma))| \leq |t_\mu| $. A computation similar to the one in the proof of Theorem \ref{thm:sysT11} gives the desired inequality.
	\end{itemize} 
	\end{proof}

	 \subsection{Four-holed sphere}

	\subsubsection{Markoff maps coming from hyperbolic structures}\
	
	Let $S$ be a four-holed sphere and let $\Ut = (a_1,a_2,a_3,a_4) \in [0, +\infty[^4$.
	
	When $a_i \geq 2$ for all $i \in \{1, 2, 3, 4 \}$, we can endow $S$ with an hyperbolic structure with geodesic boundaries or cusps such that the length of the boundaries are given by $l_i= l (a_i)$ for $i = 1, 2 , 3, 4$ (if $l_i = 0$, then it's a cusp). The equivalence class of the holonomy representation of such a structure is an element of $\mathfrak{X}_{\Ut}$.

	We saw in Section \ref{s:char} that this representation corresponds to a $\Um$-Markoff map with $\Um = GT(\Ut)$. As $a_i \geq 2$, we have naturally that $GT(a_1, a_2, a_3, a_4) \in U$.  Hence, such a Markoff map takes values in $[2 , +\infty[$, as all simple closed curves are sent to hyperbolic elements. So, we have that any Markoff map constructed from the holonomy representation of an hyperbolic structure has real positive image.
	
		A similar reasoning can be made if we replace one or several boundaries of the sphere by conical singularities of angle $0 < \theta_i < \pi$. This corresponds to the case where $a_i \in [0, 2[$ and in that case we have $\theta_i = 2 \cos^{-1} \left( \frac{a_i}{2} \right)$. One can still endow $S$ with an incomplete hyperbolic structure with geodesic boundaries, cusps and conical singularities. The holonomy representation of such a structure is not discrete and faithful, but all the simple closed curves are sent to hyperbolic elements, and hence the corresponding Markoff map remains positive. 
		
		So we can infer Theorem \ref{introthm:sys2}.(4) as a direct consequence of Theorem \ref{thm:sink4pos}. Moreover, we can also consider other type of boundaries (conical singularities and cusps) to obtain the following :

	\begin{Theorem}
	Let $S$ be an hyperbolic structure on a sphere with $b$ geodesic boundaries, $c$ cusps and $d$ conical singularities such that $b+c+d = 4$, and let $\gamma_1 , \dots , \gamma_4$ be the representatives of curves around them. 
	
	For $i \in \{ 1 , \dots , 4 \}$ we set	
	$$a_i = \begin{cases} 2 \cosh (\frac{l_i}{2}) & \mbox{ if } \gamma_i \mbox{ corresponds to a geodesic boundary of length } l_i \\
	2 & \mbox{ if } \gamma_i \mbox{ corresponds to a cusp}
	\\ 2 \cos (\frac{\theta_i}{2}) & \mbox{ if } \gamma_i \mbox{ is a conical singularity of angle } \theta_i \end{cases}$$
	Finally, let $(\lambda_1 , \lambda_2 , \lambda_3 , s) = GT ( a_1 , a_2 , a_3 , a_4)$. Then  
	
	$$\mathrm{sys} (S) \leq 2 \cosh^{-1} (t/2)$$
	where $t$ is the solution of $t^3-3t^2-(\lambda_1+\lambda_2+\lambda_3)t-s = 0$.
\end{Theorem}

	\subsubsection{Quasi-Fuchsian representations}\

	We can consider Quasi-Fuchsian representations of a four-punctured sphere, where each boundary is sent to a parabolic element. In that case we can assume without loss of generalities that $a=b=c=d=2$. The Markoff equation becomes :
	$$x^2+y^2+z^2-xyz +8x + 8y + 8z = -28$$
	We define an auxillary map $\hat \phi : \Omega \rightarrow \C$ by $ \hat \phi (X) = \phi (X) + 2$. This gives a map with properties that are slightly different from the initial Markoff map. If we denote $\hat x = \hat \phi (X)$ and so on,  the vertex relation and the edge relation become :
	\begin{equation} (\hat x + \hat y +\hat z)^2  = \hat x \hat y \hat z \end{equation}
	\begin{equation}\dfrac{\hat x + \hat y + \hat z}{\hat x \hat y} + \dfrac{\hat x + \hat y +\hat z'}{\hat x \hat y} = 1\end{equation}
	
	This new map is not a Markoff map, but we can still consider the orientation of the edges in $E (\mathcal{T})$ given by the modulus of this map. Hence we can still consider sinks for the map $\hat \phi$ in this context. 
	
	\begin{Lemma}
		Let $X, Y, Z$ be three regions meeting at a vertex and assume that is is a sink for the map $\hat \phi$. Then 
		$$ \min \{ |\hat x |, | \hat y | , | \hat z |\} \leq 9 $$
	\end{Lemma}
	
	\begin{proof} 
	
	Let $p = \Re \left( \dfrac{\hat x + \hat y + \hat z}{\hat z \hat y} \right), q = \Re \left( \dfrac{\hat x + \hat y + \hat z}{\hat x \hat z} \right), r = \Re \left( \dfrac{\hat x + \hat y + \hat z}{\hat x \hat y} \right)$, and assume without loss of generalities that $p \geq q \geq r$. At a sink, we have $p,q,r \leq \frac 12$. The vertex equation states that $p+q+r = 1$ and hence we know that $p \geq \frac 13$. Similarly, we get that $q \geq \frac 12 (1-p)$. So we have 
	$$pq \geq \frac 12 p (1-p) \geq \frac 12 \frac 13 \frac 23 = \frac 19$$
	
	Now we have that 
	$$\frac{1}{| \hat z|} =\left| \left( \dfrac{\hat x + \hat y + \hat z}{\hat z \hat y} \right)  \left( \dfrac{\hat x + \hat y + \hat z}{\hat x \hat z} \right) \right| \geq pq \geq \frac 19$$
	Which means that $|\hat z | \leq 9$.
	\end{proof}

	From this we can deduce the systolic inequality of Theorem \ref{introthm:sys2}.(5).
	
	\begin{Proposition}
		Let $\rho$ denote a Quasi-Fuchsian representation for a four-punctured sphere and $X_\rho$ the corresponding hyperbolic $3$-manifold, then
		\begin{equation} \rm{sys} (X_\rho ) \leq 2 \rm{arccosh} \left( \frac 72 \right) \end{equation}
		In particular the maximum of the systole function over the moduli space of all hyperbolic four-punctured sphere is $2 \rm{arccosh} \left( \frac 72 \right)$.
	\end{Proposition}

	 \begin{proof}

	 the previous Lemma implies that there exists a simple closed curve such that $|2 + 2 \cosh (\frac L2 ) | < 9$. As $x = 2 \cosh \left( \frac L2 \right)$ with $L = l+i\theta$ the complex translation length of $\rho (X)$, we can infer that 
	$$\left| 2+2\cosh \left( \frac L2 \right) \right| = 4 \left| \cosh^2 \left( \frac L4 \right) \right| < 4 \left| \cosh^2 \left( \frac l4 \right) \right|.$$
	Finally, we obtain :
	$$|l| \leq 4 \rm{arccosh} \left( \frac 32 \right) = 2 \rm{arccosh} \left( \frac 72 \right) $$
	
	To prove equality, we consider the Markoff triple $(7,7,7)$, which is a $(8,8,8,-28)$-Markoff triple, which is a sink in the corresponding Markoff map. 
	\end{proof}

\subsection{Closed non-orientable surface of genus 3}\

We end this section with results concerning the trace systole of representations of $\pi_1 (N_3)$ that will allow us to deduce Theorem \ref{introthm:sys2}.(6).

%\subsubsection{Trace systole}

		\begin{Theorem}\label{thm:tysN3}
			We have :
			
			$$ \mathrm{Tys} (N_3) = \sqrt{3+\sqrt{17}}$$
		
			In addition, for any $\rho : \pi_1 (N_3) \rightarrow \rm{SL} (2, \C)$, there exists a $1$-sided simple closed curve such that $|\tr (\rho (\gamma))| \leq \mathrm{Tys} (N_3)$.
		\end{Theorem}
	
	\begin{proof}
		Let $\rho$ be such a representation. Recall that if $d  \neq 0$, then $(\frac{cd}{2},  \frac{ad}{2} , \frac{bd}{2})$ is a $(0,0,0,d^2)$-Markoff triple and hence corresponds to a $d^2$-Markoff map, denoted $\phi$.
		
		Assume that $|d| > \sqrt{3+\sqrt{17}}$. From Theorem \ref{thm:sink1}, there exists an element $X \in \Omega$ such that $|\phi (X)| \leq | t_{d^2} |$. 
		
		The maximal value of $t_{\mu}$ for $\mu \in \C$ with $|\mu| = r$ occurs for $\mu = -r$. This means that $|t_{d^2} | \leq |t_{-|d|^2}|$. Moreover, as the function $\mu \rightarrow t_\mu$ is decreasing on $\R_{<0}$, we get that $|t_{d^2} | < t_{-3-\sqrt{17}}$.
		
		A simple computation gives that $t_{-3-\sqrt{17}} = \dfrac{3+\sqrt{17}}{2}$. And hence we have 
		$$|\phi(X)| < \dfrac{3+\sqrt{17}}{2}.$$
		 So there exists a $1$-sided simple closed curve $\gamma$ such that $\phi (X) = \tr (\rho (\gamma)) \frac{d}{2}$. Which means that 
		$$\tr (\rho (\gamma)) \leq \dfrac{3+\sqrt{17}}{2} \dfrac{2}{\sqrt{3+\sqrt{17}}} < \sqrt{3+\sqrt{17}}$$
		and this ends the proof of the inequality.
		
		The representation given by the character $(it,it,it,it) \in \mathfrak{N}$ with $t = \sqrt{3+\sqrt{17}}$ satisfies $\rm{tys} (\rho) = t$, which proves that the constant $t$ is optimal.
	\end{proof}

	%\subsubsection{Twisted I-bundles}

	When one restricts to quasi-fuchsian representation of $N_3$, we can recover Theorem \ref{introthm:sys2}.(6).
	
	\begin{Proposition}
		Let $\rho$ be a Quasi-Fuchsian representation for the surface $N_3$, and $X_\rho$ the corresponding hyperbolic $3$-manifold. Then 
		$$\rm{sys} (X_\rho) \leq \cosh^{-1} \left( \dfrac{5+\sqrt{17}}{2} \right) $$
	\end{Proposition}

	\begin{proof}
	From the previous theorem, there exists a $1$-sided simple closed cuve such that $|\tr (\rho (\gamma))|^2 \leq  3 + \sqrt{17}$. As we have $|\tr (\rho(\gamma)) |^2 = \left| 2 \sinh \left( \frac{l_\gamma (X)}{2} \right) \right| ^2 = | 2 \cosh (l_\gamma (X) ) -2 |$ we get that 
	$$\cosh (l_\gamma (X)) \leq \dfrac{5+\sqrt{17}}{2}$$
	
	\end{proof}

Note that this inequality was already determined by Gendulphe  \cite{gen_pay} in the case of hyperbolic surfaces.

\section{Trace systole in non-Fuchsian components of $\Sigma_2$}\label{s:bow}

In this last section, we apply our results to show that a representation of the genus 2 surface with euler class $\pm 1$ sends a simple closed curve on a non-hyperbolic element. This result was already proven by Marche and Wolff \cite{mar-wol}, but their proof is using the explicit value of the Bers constant in genus 2 and results on domination of non-Fuchsian representation by Fuchsian ones, combined with some computations in hyperbolic geometry. The proof we give is completely independent and we hope that it can be generalized in higher genus.

	\begin{Theorem}
		Let $\rho : \pi_1 (\Sigma_2) \rightarrow \mathrm{PSL} (2 , \R)$ be a representation with Euler class $\pm 1$. Then there exists a simple closed curve $\gamma \in \pi_1 (\Sigma_2)$ such that $|\tr (\rho (\gamma)) | \leq 2$.
	\end{Theorem}
	
	\begin{proof} Assume by contradiction that there is a representation such that each simple closed curve is sent to an hyperbolic element in $\mathrm{PSL} (2 , \R)$.% We consider a naive trace reduction algorithm aa follows.

	Recall that for each curve $\gamma$ in a pants decomposition $\mathcal{P}$ of $\Sigma_2$, we can consider the subsurface $\Sigma_{\gamma, \mathcal{P}}$ obtained by gluing back the one or two pants containing $\gamma$ along $\gamma$. This subsurface is either a four-holed sphere (if two distinct pants contain $\gamma$) or a one-holed torus (if $\gamma$ only appears in one pants).
	
	Assume first that there exists a pants decomposition $\mathcal{P}$ of $\Sigma_2$ with the following property : Each curve $\gamma$ of $\mathcal{P}$ realizes the trace systole of the restriction of the representation $\rho$ to the subsurface $\Sigma_{\gamma, \mathcal{P}}$ with the boundary data determined by the other curves of the pants decomposition. We call a pants decomposition satisfying this property a \emph{locally minimal} pants decomposition.
	
	%Start from any pants decomposition $\mathcal{P}_0$ of the surface and choose the largest curve $\gamma_0$ of the decomposition. Consider the surface obtained by gluing the one or two pair of pants that share $\gamma_0$ as a boundary (which is a one-holed torus or a pair-of-pants). If $\gamma_0$ is a systole of that subsurface, then we pick the second largest curve in $\mathcal{P}_0$ and replace $\gamma_0$ with this one. Other wise, we have the systole of that subsurface, and this gives a new pants decomposition $\mathcal{P}_1$, and one can continue the same process with $\mathcal{P}_1$. This gives rise to a sequence $\mathcal{P}_n$ of pants decomposition such that each new pants decomposition arising in this way has one of its curve that is strictly shorter than the previous one. 
%	
%	As all the simple closed curves are sent to hyperbolic elements, only two situations can arise :
%	
%	\begin{enumerate}
%		\item The sequence terminates after a finite number of steps. In other words, there is a pants decomposition $\mathcal{P}_n$ such that every curve of $\mathcal{P}_n$ is a systole of the subsurface (one holed torus or four-holed sphere) defined by it.
%		\item The sequence is infinite.
%	\end{enumerate}
%	
%	We first find the contradiction in case $(1)$. We denote the final pants decomposition $\mathcal{P} = \mathcal{P}_n$. 
	
	Let $\a$ and $\d$ be the two curves of minimal length of this pants decomposition. If $\mathcal{P}$ contains a separating curve $\gamma$, then this curves separates $\Sigma_2$ into two one-holed torus $T_1$ and $T_2$ glued along $\gamma$. Let $\rho_1$ and $\rho_2$ be the restriction of the representation $\rho$ to each of these subsurface. By additivity of the Euler class, we have that one of these representations has Euler class $0$, for example $\rho_1$. By hypothesis, $\rm{tys} (\rho_1) \geq 2$, which means that $\tr (\rho (\g)) \geq 18$ by Theorem \ref{thm:sysT11NF}. But for $|\tr (\rho (\g))| \geq 18$, we know for both representations $\rho_1$ and $\rho_2$, there exists a curve whose trace is smaller than $|\tr (\rho (\g))|$ and hence $\gamma$ cannot be equal to $\alpha$ or $\beta$, and hence the curves $\a$ and $\d$ are necessarily non-separating.

We consider the four-holed sphere $S$, obtained by cutting $\Sigma_2$ along $\a$ and $\d$, and denote by $\rho'$ the representation of $\rho$ restricted to $S$. As the relative Euler class of the representation $\rho'$ restricted to $S$ is $-1$, we can assume without loss of generality that the boundary traces of this representation are given by $(a,a,d,-d)$, with $a = |\tr (\rho (\a))|$ and $d = |\tr (\rho (\d))|$. As the third curve of the pants decomposition is sent to an hyperbolic element and is the systole of the representation $\rho'$ we can consider a triple of simple closed curve $(X, Y , Z)$ in $S$, such that the corresponding vertex is a sink of the associated Markoff map, and such that $X$ is the separating curve. Without loss of generalities, up to changing signs of the generators, we can assume that $2 < a < d < z < |y|$.

The curve $X$ separates $\Sigma_2$ into two one-holed torus and the induced representations have Euler class $-1$ and $0$. Which means that $|x| > \max (a^3-3a^2+2 , d^3+3d^2-2) = \lambda$. And $(x,y,z)$ is a sink so that $|y| < |xz+y|$ and $|x| < |yz+x-(a^2-d^2)|$.

We are going to prove that $x^2+y^2+z^2+xyz-(a^2-d^2)x-(a^2-2)(d^2-2) \neq 0$ which will give a contradiction. To do so, we consider $z$ fixed and define the function :
$$f(x,y) = x^2+y^2+z^2+xyz-(a^2-d^2)x-(a^2-2)(d^2-2)$$

If $xy >0$, then $x> a^2-4$, and $d > 2$, so we have $x-(a^2-d^2) >0$. And similarly $x> (d^2-2)$ and $yz > a^2 > a^2-2$ so $xyz > (a^2-2)(d^2-2))$ . Combining these arguments we have: 
$$f(x,y) = (x^2 - (a^2 - d^2)x) +(xyz - (a^2-2)(d^2-2))) + y^2 + z^2 > 8$$
And this gives a contradiction in the first case.

So we can assume that $xy < 0$. In that case, the sink inequalities become $|2y|<|xz|$ and $|2x|<|yz-(a^2-d^2)|$. Without loss of generality, up to changing signs of the generators, we restrict our study of the function $f(x,y)$ on the domain defined by $y>0$, and the inequalities given by $2y < -xz$ and $-2x < yz - (a^2 -d^2)$. This is a convex domain whose boundary is a union of lines, and the partial derivatives of $f$ are given by :
$$\frac{\partial f }{\partial x} = 2x + yz - (a^2-d^2) >0 , \quad  \frac{\partial f }{\partial y} = 2y + xz  <0 $$

So it suffices to prove that $f(x,y)< 0 $ on the lower right corner of the domain. This corner has coordinates 
	$$(x_0 , y_0) = \left( \lambda , \max \left( z , \frac{2\lambda+(a^2-d^2)}{z}\right) \right) $$ 

We distinguish two cases depending on the value of the maximum in the second coordinate :

Case 1 : If $y_0 = z$, in which case we have $z^2 > 2 \lambda + (a^2-d^2)$.

\begin{align*}
	f(\lambda , z) & = \lambda^2+2z^2-\lambda z^2+(a^2-d^2)\lambda - (a^2-2)(d^2-2) \\
		& = -(\lambda - 2)(z^2-\lambda) + 2 \lambda + (a^2-d^2)\lambda - (a^2-2)(d^2-2) \\
		& < - (\lambda - 2) (\lambda + (a^2-d^2)) + 2\lambda + (a^2-d^2)\lambda - (a^2-2)(d^2-2) \\
		& < -(\lambda-4)+2(a^2-d^2)-(a^2-2)(d^2-2) < 0
\end{align*}

Case 2 : If $y_0 = \frac{2\lambda+(a^2-d^2)}{z}$, in which case we have $z^2 < 2 \lambda + (a^2-d^2)$

\begin{align*}
	f\left(\lambda , \frac{2\lambda+(a^2-d^2)}{z} \right)   %&  \lambda^2+z^2 +\left(\frac{2\lambda+(a^2-d^2)}{z}\right)^2 -\lambda (2\lambda+(a^2-d^2))  \\
		%& +(a^2-d^2)\lambda - (a^2-2)(d^2-2) \\
		 =& - \lambda^2 + z^2+ \left( \frac{2\lambda+(a^2-d^2)}{z}\right)^2 - (a^2-2)(d^2-2)\\
		 <& - \lambda^2+2\lambda + (a^2-d^2) + (2\lambda + (a^2-d^2)) \frac{2\lambda +(a^2-d^2)}{z^2} \\
		 	& - (a^2-2)(d^2-2) \\
		 <& - \lambda^2 +2 (2\lambda+(a^2-d^2))-  (a^2-2)(d^2-2) \\
		 <& - \lambda^2+4\lambda + 2 (a^2-d^2)-(a^2-2)(d^2-2) \\
		 <& \ 0
\end{align*}

So in both cases, we get the desired contradiction that $f(x,y) \neq 0$.

\medskip
%
%To finish the proof, we just need to treat the case $(2)$ of an infinite sequence of pants decompositions.

To finish the proof, we need to consider the case where there does not exist a locally minimal pants decomposition. In that case, we consider any decomposition $\mathcal{P}_0$ and construct a sequence of pants decomposition $(\mathcal{P}_n)$ with the following property : Each pants decomposition $\mathcal{P}_{n+1}$ is constructed from $\mathcal{P}_n$ by changing one of the curve of $\mathcal{P}_n$ for a curve with a smaller trace (which can always be done as at least one of the curve does not realize the trace systole of the subsurface it defines), and keeping the two other curves of $\mathcal{P}_n$ unchanged.

 In that case, the sequences of absolute values of traces of the curves in the pants decomposition are decreasing and are bounded below by $0$, so they converge. So we can infer that for any $\varepsilon > 0$ there is exists $n$ such that the pants decomposition $\mathcal{P}_n$ has the property that each pants curve is within $\varepsilon$ of the trace systole of its corresponding subsurface.

For $\varepsilon$ sufficiently small, this is sufficient to infer that :
\begin{itemize}
	\item The two shortest curves $\a$ and $\d$ of that pants decomposition are non-separating. 
	\item In the four-holed sphere obtained by cutting along $\a$ and $\d$, we have a sink $(X,Y,Z)$ such that $X$ corresponds to a separating curve, and $2 < a < d < z+\varepsilon < |y| + \varepsilon$. 
\end{itemize}
So we can reproduce the same computations as before with an additional $\varepsilon$, and we can see that when $\varepsilon$ is small enough, the final strict inequalities still hold.

\end{proof}

	\bibliographystyle{plain}
\bibliography{sample2}

\end{document}